\documentclass[11pt]{article}
\usepackage[T1]{fontenc}
\usepackage{amsthm, amsmath, amssymb, amsfonts, url, booktabs, tikz, setspace, fancyhdr, bm}
\usepackage{hyperref}
\usepackage{geometry}
\usepackage{enumerate}
\usepackage[shortlabels]{enumitem}
\usepackage[babel]{microtype}
\usepackage[english]{babel}
\usepackage[capitalise]{cleveref}
\usepackage{cite}
\usepackage{comment}
\usepackage{bbm}
\usepackage{csquotes}
\usepackage{mathabx}
\usepackage{tikz}
\usepackage{graphicx}
\usepackage{float}
\usepackage{diagbox}
\counterwithin{figure}{section}

\newtheorem{theorem}{Theorem}[section]

\newtheorem{lemma}[theorem]{Lemma}
\newtheorem{problem}[theorem]{Problem}

\newtheorem{conj}[theorem]{Conjecture}
\newtheorem{claim}[theorem]{Claim}

\theoremstyle{definition}
\newtheorem{defn}[theorem]{Definition}
\newtheorem*{defn-non}{Definition}

\newcounter{propcounter}

\newenvironment{poc}{\begin{proof}[Proof of claim]
}{\end{proof}}

\DeclareMathOperator*{\E}{{\mathbb E}}
\newcommand{\C}[1]{{\protect\mathcal{#1}}}

\usepackage{todonotes} 

\newcommand{\del}{\delta}
\newcommand{\eps}{\varepsilon}
\newcommand{\al}{\alpha}

\newcommand{\ex}{\mathrm{ex}}
\newcommand{\BE}{\mathrm{BE}}

\newcommand{\sph}{\textsf{S}}

\title{Generalized Ramsey--Tur\'an density for cliques}
\author{Jun Gao \thanks{Extremal Combinatorics and Probability Group (ECOPRO), Institute for Basic Science (IBS), Daejeon, South Korea.
Supported by IBS-R029-C4.
 Emails: jungao@ibs.re.kr, hongliu@ibs.re.kr. }
\and Suyun Jiang\thanks{School of Artificial Intelligence, Jianghan University, Wuhan, Hubei, China, and Extremal Combinatorics and Probability Group (ECOPRO), Institute for Basic Science (IBS), Daejeon, South Korea. Supported by National Natural Science Foundation of China (11901246) and China Scholarship Council and IBS-R029-C4.
Email: jiang.suyun@163.com. }
\and
Hong Liu \footnotemark[1]
\and Maya Sankar \thanks{Department of Mathematics, Stanford University, Stanford, CA 94305, USA. Email: mayars@stanford.edu. Research supported by NSF GRFP Grant DGE-1656518 and a Hertz Fellowship.}
}

\begin{document}

\maketitle
\begin{abstract}
We study the generalized Ramsey--Tur\'an function $\mathrm{RT}(n,K_s,K_t,o(n))$, which is the maximum possible number of copies of $K_s$ in an $n$-vertex $K_t$-free graph with independence number $o(n)$. The case when $s=2$ was settled by Erd{\H{o}}s, S{\'o}s, Bollob{\'a}s, Hajnal, and Szemer\'{e}di in the 1980s. We combinatorially resolve the general case for all $s\ge 3$, showing that the (asymptotic) extremal graphs for this problem have simple (bounded) structures. In particular, it implies that the extremal structures follow a periodic pattern when $t$ is much larger than $s$. Our results disprove a conjecture of Balogh, Liu, and Sharifzadeh and show that a relaxed version does hold.
\end{abstract}

\section{Introduction}

Ramsey theory, initially explored by Ramsey \cite{ramsey1930} in 1930, stands as a pivotal branch of combinatorics. It seeks to tackle a fundamental question: what is the minimum size required to guarantee the existence of a well-defined substructure within a larger, often chaotic, set or system? One of the most renowned results in Ramsey theory is Ramsey's theorem, which asserts that if $n$ is large enough in terms of $k$, then no matter how one colors the edges of a complete graph of order $n$ using two colors, there will always exist a monochromatic complete subgraph $K_k$.

In 1941, Tur\'an \cite{turan1941on} proposed and solved the following problem: what is the maximum number of edges that a graph $G$ of order $n$ can have without containing a complete graph $K_k$? He also proved that the value is attained only by the balanced complete $(k-1)$-partite graph, now known as the \emph{Tur\'an graph} $T_{k-1}(n)$. Subsequently, a new branch of extremal combinatorics named after him emerged: Tur\'an-type problems. Formally, we define the generalized Tur\'an function $\ex(n, H_1, H_2)$ as the maximum possible number of copies of $H_1$ in an $n$-vertex $H_2$-free graph. There has been extensive research on this function. When $H_1\cong K_2$, Erd\H{o}s, Stone, and Simonovits (see \cite{erdos1946linear,erdos1966limit}) gave an asymptotically satisfactory solution for all graphs $H_2$, and Erd\H{o}s \cite{erdos1962number} additionally determined $\ex(n, K_s, K_t)$ for all $t>s\geq 3$. More recently, Alon and Shikhelman \cite{alon2016many} systematically studied $\ex(n, H_1, H_2)$ for other graphs $H_1$, and there have been a number of results in this direction (see e.g.\ \cite{Gishboliner2020Generalized,ma2020some,Morrison2023every,Beke2023odd,gao2023theta}).

In this paper, we study the following extremal quantity which mixes Ramsey theory with Tur\'an-type problems. Define the \emph{generalized Ramsey--Tur\'an number} $\mathrm{RT}(n,H_1,H_2,\ell)$ to be the maximum number of copies of $H_1$ in an $n$-vertex $H_2$-free graph $G$ with independence number $\alpha(G)< \ell$. We remark that the existence of such a graph $G$ is controlled by the Ramsey number $R(H_2,K_{\ell})$, which is defined to be the least $N$ such that every graph $G$ on $N$ vertices contains either a subgraph isomorphic to $H_2$ or an independent set of size $\ell$. This quantity is also inherently related to the generalized Tur\'an function; indeed, we have $\ex(n, H_1, H_2)=\mathrm{RT}(n,H_1,H_2,n+1)$.

This beautiful way of combining Ramsey theory with Tur\'an-type problems was first proposed in the late 1960s by S\'os \cite{sos69}, who investigated $\mathrm{RT}(n,K_2,H,\ell)$. The most studied case is when the independence number is sublinear: $\ell=o(n)$. To eliminate minor fluctuations caused by small values of $n$, one usually considers the asypmtotic behavior via the \emph{Ramsey--Tur\'an density} function,
\[
\varrho_{s}(K_t)=\lim_{\delta \to  0}\ \lim_{n \to \infty}\frac{\mathrm{RT}(n,K_s,K_t,\delta n)}{{n\choose s}}.
\]
It is not hard to see that the above limits exist. Then define $\mathrm{RT}(n,K_s,K_t,o(n))=\varrho_s(K_t){n\choose s}+o(n^s)$. We say an $n$-vertex $K_t$-free graph $G$ with $\alpha(G)=o(n)$ is an (asymptotic) extremal graph if its $K_s$-density attains $\varrho_s(K_t)$.

When $s=2$, the Ramsey--Tur\'an density has now been completely determined. It was, however, a bumpy road. In 1969, Erd{\H{o}}s and S{\'o}s \cite{erdos1969some} showed that $\varrho_2(K_{2k+1})=\frac{k-1}{k}$. The even cliques case became significantly more challenging. As a first application of the celebrated regularity lemma, Szemer{\'e}di \cite{szemeredi1972ongraphs} in 1972 proved that $\varrho_2(K_4)\le \frac{1}{4}$, and in 1976 Bollob{\'a}s and Erd{\H{o}}s \cite{Bollobas1976ramseyturan} obtained a matching lower bound $\varrho_2(K_4)\ge \frac{1}{4}$ via an astonishing geometric construction (now called the Bollob\'{a}s--Erd\H{o}s graph). Eventually, in 1983, Erd{\H{o}}s, Hajnal, S{\'o}s and Szemer\'{e}di \cite{erdo1983more} completed the picture, showing that $\varrho_2(K_{2k})=\frac{3k-5}{3k-2}$ for all $k\geq 2$. In fact, they proved a much stronger result showing that extremal graphs for $\varrho_2(K_t)$ exhibit the following periodic behavior:
\begin{itemize}
	\item[$(\star)$] Let $t=2p+r\ge 4$, where $r\in\{0,1\}$. There is an extremal graph $G$ for $\varrho_2(K_t)$ whose vertex set can be partitioned into $V_1\cup \ldots \cup V_p$ satisfying
	(i) $G[V_1,V_2]$ has edge density $\frac{r+1}{2}-o(1)$; (ii) every other $G[V_i,V_j]$ has edge density $1-o(1)$; and (iii) each $G[V_i]$ has edge density $o(1)$.
\end{itemize}
In other words, the extremal structure depends on the parity $r$ of $t$ and evolves as follows: the density of $G[V_1,V_2]$ increases as the parity $r$ increases; and whenever $t$ increases by $2$, a new part is added and joined almost completely to all previous parts (depicted in the first row of Table~\ref{tab:my_label}). For more recent developments of the $s=2$ case and related variations, we refer the interested reader to~\cite{balogh2012some,fox2015cw,balogh2015phase,Balogh2016triangle,balogh2017on,Jaehoon2019,luders2019the,liumeng2021,Liu2021geometric,Knierim2021factor,luczak2022on,hulin2023,han_hu_wang_yang_2023,Balogh2023weighted,chang2023embedding,han2024ramsey,BALOGH2024}.

Balogh, Liu, and Sharifzadeh \cite{balogh2017on} recently initiated the study of the general case $s\ge 3$, which turns out to be much more difficult and delicate than the $s=2$ case. Note that $\varrho_{s}(K_{s+1})=0$: in any $n$-vertex $K_{s+1}$-free graph with independence number $o(n)$, each copy of $K_{s-1}$ lies in $o(n)$ copies of $K_s$. Balogh, Liu, and Sharifzadeh \cite{balogh2017on} determined the first non-trivial cases $\varrho_{3}(K_{t})$ and $\varrho_{s}(K_{s+2})$, and made a conjecture predicting the general case. We find it more convenient to work with the following definition, which helps reformulate their conjecture.

\begin{defn}\label{def:type}
Given integers $b\geq a\ge 1$, a graph $G$ \emph{admits a $(b,a)$-partition} if its vertex set has a partition $V=V_1\cup\ldots\cup V_a$ satisfies the following for $b_1,\ldots,b_a\in\{\lceil\frac{b}{a}\rceil,\lfloor\frac{b}{a}\rfloor\}$ with $\sum_{i=1}^a b_i = b$:
\begin{enumerate}
    \item[(1)] For every distinct $i,j\in[a]$, $G[V_i,V_j]$ has edge density $1-o(1)$; and
    \item[(2)] For every $i\in[a]$, $V_i$ admits an equipartition $V_i^1\cup\ldots\cup V_i^{b_i}$ such that $G[V_i^j]$ has density $o(1)$ for all $j\in[b_i]$ and $G[V_i^j,V_i^{k}]$ has density $\frac{1}{2}-o(1)$ for all distinct $j,k\in[b_i]$.
\end{enumerate}
\end{defn}

For instance, if $a=b=p$ then each $b_i=1$, so an $n$-vertex graph admits a $(p,p)$-partition if and only if it has edit-distance $o(n^2)$ to the Tur\'an graph $T_p(n)$.

\begin{table}[]
    \centering
    \begin{tabular}{c|c|c|c|c|c|c|c|c}
        \toprule
        \diagbox{$s$}{$t$} & 4 & 5& 6& 7 & 8 & 9& 10 & 11 \\ \hline
        2& 
\begin{minipage}{0.085\textwidth}
\centering    
\begin{tikzpicture}[scale=0.45]
\draw[red] [line width=1.5pt] (4.5,10)-- (5.5,10);
\draw [fill=white, line width=1pt] (4.5,10) circle (0.28cm);
\draw [fill=white, line width=1pt] (5.5,10) circle (0.28cm);
\end{tikzpicture}
\end{minipage}  &
\begin{minipage}{0.085\textwidth}
\centering    
\begin{tikzpicture}[scale=0.45]
\draw [line width=1pt] (4.5,10)-- (5.5,10);
\draw [fill=white, line width=1pt] (4.5,10) circle (0.28cm);
\draw [fill=white, line width=1pt] (5.5,10) circle (0.28cm);
\end{tikzpicture}
\end{minipage} & 
\begin{minipage}{0.085\textwidth}
\centering 
\begin{tikzpicture}[scale=0.45]
\draw [line width=1pt] (1,14.5)-- (0.5,13.5);
\draw [line width=1pt] (1,14.5)-- (1.5,13.5);
\draw[red] [line width=1.5pt] (0.5,13.5)-- (1.5,13.5);
\draw [fill=white, line width=1pt] (1,14.5) circle (0.28cm);
\draw [fill=white, line width=1pt] (0.5,13.5) circle (0.28cm);
\draw [fill=white, line width=1pt] (1.5,13.5) circle (0.28cm);
\end{tikzpicture}
\end{minipage} & 
\begin{minipage}{0.085\textwidth}
\centering 
\begin{tikzpicture}[scale=0.45]
\draw [line width=1pt] (1,14.5)-- (0.5,13.5);
\draw [line width=1pt] (1,14.5)-- (1.5,13.5);
\draw [line width=1pt] (0.5,13.5)-- (1.5,13.5);
\draw [fill=white, line width=1pt] (1,14.5) circle (0.28cm);
\draw [fill=white, line width=1pt] (0.5,13.5) circle (0.28cm);
\draw [fill=white, line width=1pt] (1.5,13.5) circle (0.28cm);
\end{tikzpicture}
\end{minipage} & 
\begin{minipage}{0.085\textwidth}
\centering 
\begin{tikzpicture}[scale=0.45]
\draw [line width=1pt] (4.5,10)-- (5.5,10);
\draw[red] [line width=1.5pt] (4.5,9)-- (5.5,9);
\draw [line width=1pt] (4.5,10)-- (4.5,9);
\draw [line width=1pt] (5.5,10)-- (5.5,9);
\draw [line width=1pt] (4.5,10)-- (5.5,9);
\draw [line width=1pt] (4.5,9)-- (5.5,10);
\draw [fill=white, line width=1pt] (4.5,10) circle (0.28cm);
\draw [fill=white, line width=1pt] (5.5,10) circle (0.28cm);
\draw [fill=white, line width=1pt] (4.5,9) circle (0.28cm);
\draw [fill=white, line width=1pt] (5.5,9) circle (0.28cm);
\end{tikzpicture}
\end{minipage} & 
\begin{minipage}{0.085\textwidth}
\centering 
\begin{tikzpicture}[scale=0.45]
\draw [line width=1pt] (4.5,10)-- (5.5,10);
\draw [line width=1pt] (4.5,9)-- (5.5,9);
\draw [line width=1pt] (4.5,10)-- (4.5,9);
\draw [line width=1pt] (5.5,10)-- (5.5,9);
\draw [line width=1pt] (4.5,10)-- (5.5,9);
\draw [line width=1pt] (4.5,9)-- (5.5,10);
\draw [fill=white, line width=1pt] (4.5,10) circle (0.28cm);
\draw [fill=white, line width=1pt] (5.5,10) circle (0.28cm);
\draw [fill=white, line width=1pt] (4.5,9) circle (0.28cm);
\draw [fill=white, line width=1pt] (5.5,9) circle (0.28cm);
\end{tikzpicture}
\end{minipage}& 
\begin{minipage}{0.085\textwidth}
\centering 
\vspace{3pt}
\begin{tikzpicture}[scale=0.45]
\draw [line width=1pt] (5,11)-- (4.2,10);
\draw [line width=1pt] (5,11)-- (5.8,10);
\draw [line width=1pt] (5,11)-- (4.5,9);
\draw [line width=1pt] (5,11)-- (5.5,9);
\draw [line width=1pt] (4.2,10)-- (5.8,10);
\draw [line width=1pt] (4.2,10)-- (4.5,9);
\draw [line width=1pt] (4.2,10)-- (5.5,9);
\draw [line width=1pt] (5.8,10)-- (4.5,9);
\draw [line width=1pt] (5.8,10)-- (5.5,9);
\draw[red] [line width=1.5pt] (4.5,9)-- (5.5,9);
\draw [fill=white, line width=1pt] (4.2,10) circle (0.28cm);
\draw [fill=white, line width=1pt] (5.8,10) circle (0.28cm);
\draw [fill=white, line width=1pt] (4.5,9) circle (0.28cm);
\draw [fill=white, line width=1pt] (5.5,9) circle (0.28cm);
\draw [fill=white, line width=1pt] (5,11) circle (0.28cm);
\end{tikzpicture}
\vspace{3pt}
\end{minipage} & 
\begin{minipage}{0.085\textwidth}
\centering 
\begin{tikzpicture}[scale=0.45]
\draw [line width=1pt] (5,11)-- (4.2,10);
\draw [line width=1pt] (5,11)-- (5.8,10);
\draw [line width=1pt] (5,11)-- (4.5,9);
\draw [line width=1pt] (5,11)-- (5.5,9);
\draw [line width=1pt] (4.2,10)-- (5.8,10);
\draw [line width=1pt] (4.2,10)-- (4.5,9);
\draw [line width=1pt] (4.2,10)-- (5.5,9);
\draw [line width=1pt] (5.8,10)-- (4.5,9);
\draw [line width=1pt] (5.8,10)-- (5.5,9);
\draw [line width=1pt] (4.5,9)-- (5.5,9);
\draw [fill=white, line width=1pt] (4.2,10) circle (0.28cm);
\draw [fill=white, line width=1pt] (5.8,10) circle (0.28cm);
\draw [fill=white, line width=1pt] (4.5,9) circle (0.28cm);
\draw [fill=white, line width=1pt] (5.5,9) circle (0.28cm);
\draw [fill=white, line width=1pt] (5,11) circle (0.28cm);
\end{tikzpicture}
\end{minipage} \\ \hline
3& $\emptyset$ & 
\begin{minipage}{0.085\textwidth}
\centering 
\begin{tikzpicture}[scale=0.45]
\draw[red] [line width=1.5pt] (1,14.5)-- (0.5,13.5);
\draw[red] [line width=1.5pt] (1,14.5)-- (1.5,13.5);
\draw[red] [line width=1.5pt] (0.5,13.5)-- (1.5,13.5);
\draw [fill=white, line width=1pt] (1,14.5) circle (0.28cm);
\draw [fill=white, line width=1pt] (0.5,13.5) circle (0.28cm);
\draw [fill=white, line width=1pt] (1.5,13.5) circle (0.28cm);
\end{tikzpicture}
\end{minipage}& 
\begin{minipage}{0.085\textwidth}
\centering 
\begin{tikzpicture}[scale=0.45]
\draw [line width=1pt] (1,14.5)-- (0.5,13.5);
\draw [line width=1pt] (1,14.5)-- (1.5,13.5);
\draw[red] [line width=1.5pt] (0.5,13.5)-- (1.5,13.5);
\draw [fill=white, line width=1pt] (1,14.5) circle (0.28cm);
\draw [fill=white, line width=1pt] (0.5,13.5) circle (0.28cm);
\draw [fill=white, line width=1pt] (1.5,13.5) circle (0.28cm);
\end{tikzpicture}
\end{minipage} & 
\begin{minipage}{0.085\textwidth}
\centering 
\begin{tikzpicture}[scale=0.45]
\draw [line width=1pt] (1,14.5)-- (0.5,13.5);
\draw [line width=1pt] (1,14.5)-- (1.5,13.5);
\draw [line width=1pt] (0.5,13.5)-- (1.5,13.5);
\draw [fill=white, line width=1pt] (1,14.5) circle (0.28cm);
\draw [fill=white, line width=1pt] (0.5,13.5) circle (0.28cm);
\draw [fill=white, line width=1pt] (1.5,13.5) circle (0.28cm);
\end{tikzpicture}
\end{minipage} & 
\begin{minipage}{0.085\textwidth}
\centering 
\begin{tikzpicture}[scale=0.45]
\draw [line width=1pt] (4.5,10)-- (5.5,10);
\draw[red] [line width=1.5pt] (4.5,9)-- (5.5,9);
\draw [line width=1pt] (4.5,10)-- (4.5,9);
\draw [line width=1pt] (5.5,10)-- (5.5,9);
\draw [line width=1pt] (4.5,10)-- (5.5,9);
\draw [line width=1pt] (4.5,9)-- (5.5,10);
\draw [fill=white, line width=1pt] (4.5,10) circle (0.28cm);
\draw [fill=white, line width=1pt] (5.5,10) circle (0.28cm);
\draw [fill=white, line width=1pt] (4.5,9) circle (0.28cm);
\draw [fill=white, line width=1pt] (5.5,9) circle (0.28cm);
\end{tikzpicture}
\end{minipage} & 
\begin{minipage}{0.085\textwidth}
\centering 
\begin{tikzpicture}[scale=0.45]
\draw [line width=1pt] (4.5,10)-- (5.5,10);
\draw [line width=1pt] (4.5,9)-- (5.5,9);
\draw [line width=1pt] (4.5,10)-- (4.5,9);
\draw [line width=1pt] (5.5,10)-- (5.5,9);
\draw [line width=1pt] (4.5,10)-- (5.5,9);
\draw [line width=1pt] (4.5,9)-- (5.5,10);
\draw [fill=white, line width=1pt] (4.5,10) circle (0.28cm);
\draw [fill=white, line width=1pt] (5.5,10) circle (0.28cm);
\draw [fill=white, line width=1pt] (4.5,9) circle (0.28cm);
\draw [fill=white, line width=1pt] (5.5,9) circle (0.28cm);
\end{tikzpicture}
\end{minipage}& 
\begin{minipage}{0.085\textwidth}
\centering 
\vspace{3pt}
\begin{tikzpicture}[scale=0.45]
\draw [line width=1pt] (5,11)-- (4.2,10);
\draw [line width=1pt] (5,11)-- (5.8,10);
\draw [line width=1pt] (5,11)-- (4.5,9);
\draw [line width=1pt] (5,11)-- (5.5,9);
\draw [line width=1pt] (4.2,10)-- (5.8,10);
\draw [line width=1pt] (4.2,10)-- (4.5,9);
\draw [line width=1pt] (4.2,10)-- (5.5,9);
\draw [line width=1pt] (5.8,10)-- (4.5,9);
\draw [line width=1pt] (5.8,10)-- (5.5,9);
\draw[red] [line width=1.5pt] (4.5,9)-- (5.5,9);
\draw [fill=white, line width=1pt] (4.2,10) circle (0.28cm);
\draw [fill=white, line width=1pt] (5.8,10) circle (0.28cm);
\draw [fill=white, line width=1pt] (4.5,9) circle (0.28cm);
\draw [fill=white, line width=1pt] (5.5,9) circle (0.28cm);
\draw [fill=white, line width=1pt] (5,11) circle (0.28cm);
\end{tikzpicture}
\vspace{3pt}
\end{minipage} & 
\begin{minipage}{0.085\textwidth}
\centering 
\begin{tikzpicture}[scale=0.45]
\draw [line width=1pt] (5,11)-- (4.2,10);
\draw [line width=1pt] (5,11)-- (5.8,10);
\draw [line width=1pt] (5,11)-- (4.5,9);
\draw [line width=1pt] (5,11)-- (5.5,9);
\draw [line width=1pt] (4.2,10)-- (5.8,10);
\draw [line width=1pt] (4.2,10)-- (4.5,9);
\draw [line width=1pt] (4.2,10)-- (5.5,9);
\draw [line width=1pt] (5.8,10)-- (4.5,9);
\draw [line width=1pt] (5.8,10)-- (5.5,9);
\draw [line width=1pt] (4.5,9)-- (5.5,9);
\draw [fill=white, line width=1pt] (4.2,10) circle (0.28cm);
\draw [fill=white, line width=1pt] (5.8,10) circle (0.28cm);
\draw [fill=white, line width=1pt] (4.5,9) circle (0.28cm);
\draw [fill=white, line width=1pt] (5.5,9) circle (0.28cm);
\draw [fill=white, line width=1pt] (5,11) circle (0.28cm);
\end{tikzpicture}
\end{minipage} \\ \hline
4&  & $\emptyset$ & 
\begin{minipage}{0.085\textwidth}
\centering 
\begin{tikzpicture}[scale=0.45]
\draw[red] [line width=1.5pt] (4.5,10)-- (5.5,10);
\draw[red] [line width=1.5pt] (4.5,9)-- (5.5,9);
\draw[red] [line width=1.5pt] (4.5,10)-- (4.5,9);
\draw[red] [line width=1.5pt] (5.5,10)-- (5.5,9);
\draw[red] [line width=1.5pt] (4.5,10)-- (5.5,9);
\draw[red] [line width=1.5pt] (4.5,9)-- (5.5,10);
\draw [fill=white, line width=1pt] (4.5,10) circle (0.28cm);
\draw [fill=white, line width=1pt] (5.5,10) circle (0.28cm);
\draw [fill=white, line width=1pt] (4.5,9) circle (0.28cm);
\draw [fill=white, line width=1pt] (5.5,9) circle (0.28cm);
\end{tikzpicture}
\end{minipage}& 
\begin{minipage}{0.085\textwidth}
\centering 
\begin{tikzpicture}[scale=0.45]
\draw [line width=1pt] (4.5,10)-- (5.5,10);
\draw [line width=1pt] (4.5,9)-- (5.5,9);
\draw[red] [line width=1.5pt] (4.5,10)-- (4.5,9);
\draw[red] [line width=1.5pt] (5.5,10)-- (5.5,9);
\draw [line width=1pt] (4.5,10)-- (5.5,9);
\draw [line width=1pt] (4.5,9)-- (5.5,10);
\draw [fill=white, line width=1pt] (4.5,10) circle (0.28cm);
\draw [fill=white, line width=1pt] (5.5,10) circle (0.28cm);
\draw [fill=white, line width=1pt] (4.5,9) circle (0.28cm);
\draw [fill=white, line width=1pt] (5.5,9) circle (0.28cm);
\end{tikzpicture}
\end{minipage} & 
\begin{minipage}{0.085\textwidth}
\centering 
\begin{tikzpicture}[scale=0.45]
\draw [line width=1pt] (4.5,10)-- (5.5,10);
\draw[red] [line width=1.5pt] (4.5,9)-- (5.5,9);
\draw [line width=1pt] (4.5,10)-- (4.5,9);
\draw [line width=1pt] (5.5,10)-- (5.5,9);
\draw [line width=1pt] (4.5,10)-- (5.5,9);
\draw [line width=1pt] (4.5,9)-- (5.5,10);
\draw [fill=white, line width=1pt] (4.5,10) circle (0.28cm);
\draw [fill=white, line width=1pt] (5.5,10) circle (0.28cm);
\draw [fill=white, line width=1pt] (4.5,9) circle (0.28cm);
\draw [fill=white, line width=1pt] (5.5,9) circle (0.28cm);
\end{tikzpicture}
\end{minipage} & 
\begin{minipage}{0.085\textwidth}
\centering 
\begin{tikzpicture}[scale=0.45]
\draw [line width=1pt] (4.5,10)-- (5.5,10);
\draw [line width=1pt] (4.5,9)-- (5.5,9);
\draw [line width=1pt] (4.5,10)-- (4.5,9);
\draw [line width=1pt] (5.5,10)-- (5.5,9);
\draw [line width=1pt] (4.5,10)-- (5.5,9);
\draw [line width=1pt] (4.5,9)-- (5.5,10);
\draw [fill=white, line width=1pt] (4.5,10) circle (0.28cm);
\draw [fill=white, line width=1pt] (5.5,10) circle (0.28cm);
\draw [fill=white, line width=1pt] (4.5,9) circle (0.28cm);
\draw [fill=white, line width=1pt] (5.5,9) circle (0.28cm);
\end{tikzpicture}
\end{minipage}& 
\begin{minipage}{0.085\textwidth}
\centering 
\vspace{3pt}
\begin{tikzpicture}[scale=0.45]
\draw [line width=1pt] (5,11)-- (4.2,10);
\draw [line width=1pt] (5,11)-- (5.8,10);
\draw [line width=1pt] (5,11)-- (4.5,9);
\draw [line width=1pt] (5,11)-- (5.5,9);
\draw [line width=1pt] (4.2,10)-- (5.8,10);
\draw [line width=1pt] (4.2,10)-- (4.5,9);
\draw [line width=1pt] (4.2,10)-- (5.5,9);
\draw [line width=1pt] (5.8,10)-- (4.5,9);
\draw [line width=1pt] (5.8,10)-- (5.5,9);
\draw[red] [line width=1.5pt] (4.5,9)-- (5.5,9);
\draw [fill=white, line width=1pt] (4.2,10) circle (0.28cm);
\draw [fill=white, line width=1pt] (5.8,10) circle (0.28cm);
\draw [fill=white, line width=1pt] (4.5,9) circle (0.28cm);
\draw [fill=white, line width=1pt] (5.5,9) circle (0.28cm);
\draw [fill=white, line width=1pt] (5,11) circle (0.28cm);
\end{tikzpicture}
\vspace{3pt}
\end{minipage} & 
\begin{minipage}{0.085\textwidth}
\centering 
\begin{tikzpicture}[scale=0.45]
\draw [line width=1pt] (5,11)-- (4.2,10);
\draw [line width=1pt] (5,11)-- (5.8,10);
\draw [line width=1pt] (5,11)-- (4.5,9);
\draw [line width=1pt] (5,11)-- (5.5,9);
\draw [line width=1pt] (4.2,10)-- (5.8,10);
\draw [line width=1pt] (4.2,10)-- (4.5,9);
\draw [line width=1pt] (4.2,10)-- (5.5,9);
\draw [line width=1pt] (5.8,10)-- (4.5,9);
\draw [line width=1pt] (5.8,10)-- (5.5,9);
\draw [line width=1pt] (4.5,9)-- (5.5,9);
\draw [fill=white, line width=1pt] (4.2,10) circle (0.28cm);
\draw [fill=white, line width=1pt] (5.8,10) circle (0.28cm);
\draw [fill=white, line width=1pt] (4.5,9) circle (0.28cm);
\draw [fill=white, line width=1pt] (5.5,9) circle (0.28cm);
\draw [fill=white, line width=1pt] (5,11) circle (0.28cm);
\end{tikzpicture}
\end{minipage} \\ \hline
         5&  & & $\emptyset$ & 
\begin{minipage}{0.085\textwidth}
\centering 
\begin{tikzpicture}[scale=0.45]
\draw[red] [line width=1.5pt] (5,11)-- (4.2,10);
\draw[red] [line width=1.5pt] (5,11)-- (5.8,10);
\draw[red] [line width=1.5pt] (5,11)-- (4.5,9);
\draw[red] [line width=1.5pt] (5,11)-- (5.5,9);
\draw[red] [line width=1.5pt] (4.2,10)-- (5.8,10);
\draw[red] [line width=1.5pt] (4.2,10)-- (4.5,9);
\draw[red] [line width=1.5pt] (4.2,10)-- (5.5,9);
\draw[red] [line width=1.5pt] (5.8,10)-- (4.5,9);
\draw[red] [line width=1.5pt] (5.8,10)-- (5.5,9);
\draw[red] [line width=1.5pt] (4.5,9)-- (5.5,9);
\draw [fill=white, line width=1pt] (4.2,10) circle (0.28cm);
\draw [fill=white, line width=1pt] (5.8,10) circle (0.28cm);
\draw [fill=white, line width=1pt] (4.5,9) circle (0.28cm);
\draw [fill=white, line width=1pt] (5.5,9) circle (0.28cm);
\draw [fill=white, line width=1pt] (5,11) circle (0.28cm);
\end{tikzpicture}
\end{minipage} & 
\begin{minipage}{0.085\textwidth}
\centering 
\begin{tikzpicture}[scale=0.45]
\draw[red] [line width=1.5pt] (5,11)-- (4.2,10);
\draw[red] [line width=1.5pt] (5,11)-- (5.8,10);
\draw [line width=1pt] (5,11)-- (4.5,9);
\draw [line width=1pt] (5,11)-- (5.5,9);
\draw[red] [line width=1.5pt] (4.2,10)-- (5.8,10);
\draw [line width=1pt] (4.2,10)-- (4.5,9);
\draw [line width=1pt] (4.2,10)-- (5.5,9);
\draw [line width=1pt] (5.8,10)-- (4.5,9);
\draw [line width=1pt] (5.8,10)-- (5.5,9);
\draw[red] [line width=1.5pt] (4.5,9)-- (5.5,9);
\draw [fill=white, line width=1pt] (4.2,10) circle (0.28cm);
\draw [fill=white, line width=1pt] (5.8,10) circle (0.28cm);
\draw [fill=white, line width=1pt] (4.5,9) circle (0.28cm);
\draw [fill=white, line width=1pt] (5.5,9) circle (0.28cm);
\draw [fill=white, line width=1pt] (5,11) circle (0.28cm);
\end{tikzpicture}
\end{minipage} & 
\begin{minipage}{0.085\textwidth}
\centering 
\begin{tikzpicture}[scale=0.45]
\draw [line width=1pt] (5,11)-- (4.2,10);
\draw [line width=1pt] (5,11)-- (5.8,10);
\draw [line width=1pt] (5,11)-- (4.5,9);
\draw [line width=1pt] (5,11)-- (5.5,9);
\draw [line width=1pt] (4.2,10)-- (5.8,10);
\draw[red] [line width=1.5pt] (4.2,10)-- (4.5,9);
\draw [line width=1pt] (4.2,10)-- (5.5,9);
\draw [line width=1pt] (5.8,10)-- (4.5,9);
\draw[red] [line width=1.5pt] (5.8,10)-- (5.5,9);
\draw [line width=1pt] (4.5,9)-- (5.5,9);
\draw [fill=white, line width=1pt] (4.2,10) circle (0.28cm);
\draw [fill=white, line width=1pt] (5.8,10) circle (0.28cm);
\draw [fill=white, line width=1pt] (4.5,9) circle (0.28cm);
\draw [fill=white, line width=1pt] (5.5,9) circle (0.28cm);
\draw [fill=white, line width=1pt] (5,11) circle (0.28cm);
\end{tikzpicture}
\end{minipage}& 
\begin{minipage}{0.085\textwidth}
\centering 
\vspace{3pt}
\begin{tikzpicture}[scale=0.45]
\draw [line width=1pt] (5,11)-- (4.2,10);
\draw [line width=1pt] (5,11)-- (5.8,10);
\draw [line width=1pt] (5,11)-- (4.5,9);
\draw [line width=1pt] (5,11)-- (5.5,9);
\draw [line width=1pt] (4.2,10)-- (5.8,10);
\draw [line width=1pt] (4.2,10)-- (4.5,9);
\draw [line width=1pt] (4.2,10)-- (5.5,9);
\draw [line width=1pt] (5.8,10)-- (4.5,9);
\draw [line width=1pt] (5.8,10)-- (5.5,9);
\draw[red] [line width=1.5pt] (4.5,9)-- (5.5,9);
\draw [fill=white, line width=1pt] (4.2,10) circle (0.28cm);
\draw [fill=white, line width=1pt] (5.8,10) circle (0.28cm);
\draw [fill=white, line width=1pt] (4.5,9) circle (0.28cm);
\draw [fill=white, line width=1pt] (5.5,9) circle (0.28cm);
\draw [fill=white, line width=1pt] (5,11) circle (0.28cm);
\end{tikzpicture}
\vspace{3pt}
\end{minipage} & 
\begin{minipage}{0.085\textwidth}
\centering 
\begin{tikzpicture}[scale=0.45]
\draw [line width=1pt] (5,11)-- (4.2,10);
\draw [line width=1pt] (5,11)-- (5.8,10);
\draw [line width=1pt] (5,11)-- (4.5,9);
\draw [line width=1pt] (5,11)-- (5.5,9);
\draw [line width=1pt] (4.2,10)-- (5.8,10);
\draw [line width=1pt] (4.2,10)-- (4.5,9);
\draw [line width=1pt] (4.2,10)-- (5.5,9);
\draw [line width=1pt] (5.8,10)-- (4.5,9);
\draw [line width=1pt] (5.8,10)-- (5.5,9);
\draw [line width=1pt] (4.5,9)-- (5.5,9);
\draw [fill=white, line width=1pt] (4.2,10) circle (0.28cm);
\draw [fill=white, line width=1pt] (5.8,10) circle (0.28cm);
\draw [fill=white, line width=1pt] (4.5,9) circle (0.28cm);
\draw [fill=white, line width=1pt] (5.5,9) circle (0.28cm);
\draw [fill=white, line width=1pt] (5,11) circle (0.28cm);
\end{tikzpicture}
\end{minipage} \\ 
        \bottomrule
    \end{tabular}
    \vspace{5pt}\\
    Black edges have density 1 and red edges have density 1/2.
    \caption{Conjectured periodic extremal structure}
    \label{tab:my_label}
\end{table}

\begin{conj}[\cite{balogh2017on}]\label{conj}
    Given integers $t-2\ge s\ge 3$, there is an extremal graph for $\varrho_{s}(K_t)$ which admits
    \begin{itemize}
        \item[(i)] an $(s,t-1-s)$-partition if $s+2\le t\le 2s-1$, or
        \item[(ii)] a $(\lfloor \frac{t}{2}\rfloor,t-1-\lfloor \frac{t}{2}\rfloor)$-partition if $t\ge 2s$.
    \end{itemize}
\end{conj}

The preceding conjecture can be better understood, using the language of Definition~\ref{def:type}, as follows. For every $t=2p+r\ge 4$ with $r\in\{0,1\}$, the periodic behavior of $\varrho_2(K_t)$ in $(\star)$ can be rephrased as `\emph{an extremal graph for $\varrho_2(K_t)$ admits a $(p,p-1+r)$-partition}', which is precisely the statement of Conjecture~\ref{conj}(ii) when $s=2$. Thus, Conjecture~\ref{conj} speculates that similar periodic behavior occurs at the threshold $t\ge 2s$ for all $s\ge 3$ (see Table~\ref{tab:my_label}). Further supporting this prediction, it was proved in~\cite{balogh2017on} that Conjecture~\ref{conj} holds for $s=3$.

We present infinitely many counterexamples showing that Conjecture~\ref{conj} is false in general. The smallest counterexamples we observe are $s=5$ and $t\in \{10,11\}$ (see Figure~\ref{fig:s=5}). On the positive side, we prove that the predicted periodic behavior does eventually occur when $t\gg s$ for all $s\ge 3$. 

\begin{theorem}\label{main 2}
Conjecture \ref{conj} is false when $2s\leq t\leq 2.08s$ for sufficiently large $s$. Given $t-2\geq s\geq 3$, Conjecture \ref{conj} is true if $t>s^2(s-1)/2+s$, $t=s+2$, or $s=3,4$.
\end{theorem}

Furthermore, our main result shows that a modified version of Conjecture~\ref{conj} is true (which was the motivation for Definition~\ref{def:type}). It reads as follows.

\begin{theorem}\label{main 1}
  For all integers $t-2\geq s\ge 3$, there is a family of extremal graphs for $\varrho_s(K_t)$ admitting a $(b,a)$-partition for some parameters $1\le a\leq b$ satisfying $a+b=t-1$.
\end{theorem}

Theorem~\ref{main 1} provides a detailed description of the extremal graphs for the generalized Ramsey-Tur\'an problem for cliques, showing that they have simple and bounded structures. We remark that Theorem~\ref{main 1} resolves \emph{combinatorially} the problem of determining the Ramsey--Tur\'an density $\varrho_s(K_t)$. Indeed, given $s$ and $t$, there is a bounded number of choices for $a$ (because $a\leq t-1$). Once $a$ is fixed, the structure of a graph admitting a $(b,a)$-partition is determined, so its $K_s$-density may be computed in terms of the fractions of vertices allocated to each part. Thus, Theorem~\ref{main 1} reduces determining $\varrho_s(K_t)$ to a bounded optimization problem (over $a+b=t-1$).

\paragraph{Organization.} The rest of the paper is structured as follows. The proof of~\cref{main 1} consists of two parts. We first reduce it to a more tractable problem about clique densities in weighted graphs (see~\cref{thm:reduction-to-R}) in \cref{sec:reduction}. Understanding this weighted problem is the bulk of the proof (see \cref{thm:weight:main}); we study it in \cref{sec:extremal-weighted}.
 In Section \ref{section:5}, we give the proof of \cref{main 2}. Concluding remarks are given in Section \ref{sec:concluding remarks}.

\paragraph{Notation.} We use $[n]$ to denote the finite set $\{1,2 ,\ldots,n\}$. For a vector $\textbf{\textit{u}}=(u_1 ,\ldots,u_k)\in \mathbb{R}^k$, we write $\|\textbf{\textit{u}}\|=\sqrt{\sum_{i=1}^k u_i^2}$ for its $\ell_2$-norm. Let $G=(V(G), E(G))$ be a graph. For every $U,V\subseteq V(G)$, denote by $G[U,V]$ the induced bipartite subgraph of $G$ on partite sets $U$ and $V$, and by $G[U]$ the induced subgraph of $G$ on set $U$. For convenience, we let $G-U=G[V(G)\setminus U]$.

\section{Reduction to Weighted Graphs}\label{sec:reduction}

The first step of the proof of \cref{main 1} reduces understanding Ramsey--Tur\'an density to a problem about clique density in weighted graphs.  The aim of this section will be to prove \cref{thm:reduction-to-R}, which demonstrates the equivalence between these two problems. However, before we can state \cref{thm:reduction-to-R}, we need to define our notion of weighted graphs.

\begin{defn}
A \emph{weighted graph} $R=(V,w)$ consists of a finite vertex set $V$ together with a \emph{weight function} $w:V\sqcup V^2\to[0,1]$ satisfying the following two properties. The vertex weights must sum to one, i.e., $\sum_{v\in V}w(v)=1$. Additionally, the edge weights must satisfy $w(v,v')=w(v',v)$ and $w(v,v)=0$ for any $v,v'\in V$. For $\alpha\in[0,1]$, denote by $R_{>\alpha}$ the spanning subgraph of $R$ with all edges of weight larger than $\alpha$.
\end{defn}

Intuitively, a weighted graph may be thought of as a type of graph limit with a more discrete structure than a graphon. An $r$-vertex weighted graph can also be considered to represent a large $r$-partite graph $G$ whose $i$th part $V_i$ contains a $w(i)$-fraction of the vertex set, such that each induced bipartite subgraph $G[V_i,V_j]$ is a random graph of density $w(i,j)$.

With this perspective in mind, we define subgraph densities in a weighted graph.
\begin{defn}
Let $H$ be a graph with vertex set $[s]$. The \emph{$H$-density} of a weighted graph $R$ is defined as
\[
d_H(R)
=\E_{\substack{v_1,\ldots,v_s\\\in V(R)}}\left[\prod_{ij\in E(H)}w(v_i,v_j)\right]
=\sum_{\sigma: [s]\rightarrow  V(R)}\left(\prod_{i=1}^sw(\sigma(i))\right)\left(\prod_{ij\in E(H)}w(\sigma(i),\sigma(j))\right),
\]
where vertices $v_1,\ldots,v_s\in V(R)$ are chosen independently at random according to the vertex weights of $R$.
\end{defn}

We shall show that the Ramsey--Tur\'an density $\varrho_{s}(K_t)$ is determined by the maximum possible $K_s$-density in a weighted graph avoiding the following forbidden configuration.

\begin{defn}
Let $R$ be a weighted graph and $t\in\mathbb{N}$. The \emph{weighted $t$-clique family} $\C K_t$ consists of all pairs of subsets $(S_1,S_2)$ with $S_2\subseteq S_1\subseteq V(R)$, $s_1=|S_1|,s_2=|S_2|\ge 1$ and $s_1+s_2=t$ such that $S_1$ induces a $K_{s_1}$ in $R_{>0}$ and $S_2$ induces a $K_{s_2}$ in $R_{>\frac{1}{2}}$.We say $R$ is \emph{$\C K_t$--free} if $R$ contains no such pair $(S_1,S_2)$.
\end{defn}

We can then define the $K_s$-Tur\'an density of $\C K_t$ as
$$\pi_{s}(\C K_t)=\sup\{d_{K_s}(R):~\text{$R$ is a $\C K_t$-free weighted graph}\}.$$

At this point, we may state the main result of this section.
\begin{theorem}\label{thm:reduction-to-R}
For $s,t\in\mathbb{N}$ with $2\leq s\leq t-1$, we have $\varrho_{s}(K_t)=\pi_{s}(\C K_t)$.
\end{theorem}

The upper and lower bounds of \cref{thm:reduction-to-R} will be proven in the next two subsections.

\subsection{Upper bound}

Our proof of the upper bound on \cref{thm:reduction-to-R} relies on Szemer\'edi's regularity lemma. The regularity lemma states that any graph looks $\eps$-close to a weighted graph whose number of vertices is bounded in terms of $\eps$. We recall its statement here, beginning with some preliminary definitions.

\begin{defn}
    Let $G$ be a graph and let $X, Y\subseteq V(G)$. The \emph{edge density} between $X$ and $Y$, denoted by $d(X,Y)$, is the fraction of pairs $(x,y)\in X\times Y$ that are edges of $G$. Given $\eps>0$, we say the pair $(X,Y)$ is \emph{$\eps$-regular} if, for all $X'\subseteq X$ and $Y'\subseteq Y$ with $|X'|\geq \eps|X|$ and $|Y'|\geq\eps|Y|$, we have $|d(X',Y')-d(X,Y)|<\eps$.
\end{defn}

\begin{defn}
Let $G$ be a graph. A vertex partition $V(G)=V_1\cup\cdots\cup V_r\cup V_{r+1}$ is called \emph{$\eps$-regular} if $|V_1|=\cdots=|V_r|$ and $|V_{r+1}|\leq \eps n$, and additionally at most $\eps r^2$ pairs $(V_i,V_j)$ with $i<j\leq r$ are not $\eps$-regular.
\end{defn}

\begin{theorem}[Regularity Lemma, \cite{Szemerdi1976Regular}]\label{thm: regularity}
For every small constant $\eps>0$ and integer $M_0$, there exists an integer $M=M(\eps,M_0)$ such that the following holds. Given any $n$-vertex graph $G$, there is an $\eps$-regular partition of its vertices $V(G)=V_1\cup\cdots\cup V_r\cup V_{r+1}$ such that $M_0\leq r\leq M$.
\end{theorem}

We also require the graph counting lemma. Intuitively, if a graph $G$ looks like a weighted graph $R$, then this lemma implies that the $K_s$-density of $G$ is approximately the $K_s$-density of $R$.

\begin{lemma}[Graph Counting Lemma, {\cite{Duke1995Fast}}]\label{lem:graph-counting}
Let $\eps>0$, and let $G$ be an $s$-partite graph with $V(G)=\bigcup_{i=1}^s V_s$. Suppose that the pair $(V_i,V_j)$ is $\eps$-regular for all $1\leq i< j\leq s$. Then 
$$\left|N(K_s,G)/\left(\prod_{i=1}^s|V_i|\right)-\prod_{1\leq i<j\leq s}d(V_i,V_j)\right|\leq \sqrt{\eps s^3},$$
where $N(K_s, G)$ is the number of copies of $K_s$ in $G$.
\end{lemma}

We now prove the upper bound of \cref{thm:reduction-to-R}.

\begin{theorem}\label{reduction-UB}
Let $s,t$ be integers with $2\le s< t$. For any $\del\in(0,1)$ there exists $\del'\in(0,1)$ such that the following holds.
Suppose $G$ is a $K_t$-free graph with $\al(G)\leq \del'|V(G)|$. Then there is a $\C K_t$-free weighted graph $R$ such that $d_{K_s}(G)\leq d_{K_s}(R)+4s^2\del$.
\end{theorem}

\begin{proof}
Choose $\eps<\frac 12$ small enough such that $(\del-\eps)^{t-1}>(t+1)\eps$ and let $\del'=\eps/M$, where $M=M(\eps,\frac 1\eps)$ is the constant guaranteed by the regularity lemma (\cref{thm: regularity}). Suppose $G$ is a $K_t$-free graph on $N$ vertices with $\al(G)\leq\del'N$.

Apply Theorem \ref{thm: regularity} with this value of $\eps$ to $G$. This yields an $\eps$-regular partition $V(G)=V_1\cup\cdots\cup V_r\cup V_{r+1}$ with $\frac 1\eps\leq r\leq M$. We show that a substructure similar to a weighted $t$-clique is forbidden among the edge densities $d(V_i,V_j)$.

\begin{claim}\label{reduction-UB-claim}
Suppose $S_2\subseteq S_1\subseteq[r]$ are sets of indices such that 
\begin{enumerate}[label=(\roman*)]
	\item For any distinct $i,j\in S_1$, the pair $(V_i,V_j)$ is $\eps$-regular with density $d(V_i,V_j)>\del$; and
	\item For any distinct $i,j\in S_2$, the pair $(V_i,V_j)$ has density $d(V_i,V_j)>\frac 12+\del$.
\end{enumerate}
Then $G$ contains a clique of order $|S_1|+|S_2|$. In particular, $|S_1|+|S_2|<t$.
\end{claim}

\begin{poc}
Order the elements of $S_1$ as $a_1,\ldots,a_\ell$ with the elements of $S_1\setminus S_2$ listed first. Set $k=|S_1|-|S_2|$. By \cref{lem:graph-counting}, there exists a clique $S$ of order $\ell$ in $G$ such that $|S\cap V_{a_i}|=1$ for each $i\in [\ell]$; as $G$ is $K_t$-free, it follows that $\ell<t$. Our proof follows an $\ell$-step process, where the $i$th step chooses one (if $i\leq k$) or two (if $i>k$) vertices from $V_{a_i}$ that are adjacent to all previously chosen vertices. For $0\leq i,j\leq \ell$, write $W^{(i)}$ for the common neighborhood of those vertices chosen in the first $i$ steps, and let $W_j^{(i)}=V_{a_j}\cap W^{(i)}$. In particular, $W^{(0)}=V(G)$. We will choose $2\ell-k$ vertices such that $|W_j^{(i)}|\geq(\del-\eps)|W_j^{(i-1)}|\geq(\del-\eps)^i|V_{a_j}|$ for all $0\leq i<j\leq \ell$.

On the $i$th step with $1\leq i\leq k$, choose one vertex $v_i\in W_i^{(i-1)}$ such that $W_j^{(i)}:=N(v_i)\cap W_j^{(i-1)}$ has cardinality at least $(\del-\eps)|W_j^{(i-1)}|$ for each $j>i$. To show that such a vertex $v_i$ exists, consider the sets
\[X_j=\left\{v\in W_i^{(i-1)}:|N(v)\cap W_j^{(i-1)}|<(\del-\eps)|W_j^{(i-1)}|\right\}\] 
for each $j>i$. Observe that $d\left(X_j,W_j^{(i-1)}\right)<\del-\eps<d(V_{a_i},V_{a_j})-\eps$ by construction, and that $|W_j^{(i-1)}|\geq(\del-\eps)^{i-1}|V_{a_j}|\geq(\del-\eps)^t|V_{a_j}|\geq\eps|V_{a_j}|$. Because the pair $(V_{a_i},V_{a_j})$ is $\eps$-regular, it follows that $|X_j|<\eps|V_{a_i}|$. Thus,
\[
\left|W_i^{(i-1)}-\bigcup_{j=i+1}^\ell X_j\right| > (\del-\eps)^{i-1}|V_{a_i}|-(\ell-i)\eps|V_{a_i}|\geq ((\del-\eps)^{t-1}-(t-1)\eps)|V_{a_i}|>0.
\]
It follows that there is a vertex $v_i\in W_i^{(i-1)}$ such that $|W_j^{(i)}|=|N(v_i)\cap W_j^{(i-1)}|\geq (\del-\eps)|W_j^{(i-1)}|$ for each $j>i$.

On the $i$th step with $k<i\leq\ell$, choose two adjacent vertices $v_i,v'_i\in W_i^{(i-1)}$ such that $W_j^{(i)}:=N(v_i)\cap N(v'_i)\cap W_j^{(i-1)}$ has cardinality at least $2(\del-\eps)|W_j^{(i-1)}|$ for all $j>i$. To verify that such vertices exist, set
\[X_j=\left\{v\in W_i^{(i-1)}:|N(v)\cap W_j^{(i-1)}|<\left(\frac 12+\del-\eps\right)|W_j^{(i-1)}|\right\}\]
for each $j>i$. The argument from the prior paragraph shows that
\[
\left|W_i^{(i-1)}-\bigcup_{j=i+1}^\ell X_j\right| > \left((\del-\eps)^{t-1}-(t-1)\eps\right)|V_{a_i}|>2\eps|V_{a_i}|.
\]
Noting that $|V_{a_i}|=(N-|V_{r+1}|)/r>N/2r$, we have
\[
\left|W_i^{(i-1)}-\bigcup_{j=i+1}^\ell X_j\right| > 2\eps|V_{a_i}|
>\frac{\eps N}{r}\geq\del' N\geq \al(G).
\]
It follows that there are adjacent vertices $v_i,v'_i\in W_i^{(i-1)}$ such that $N(v_i)\cap W_j^{(i-1)}$ and $N(v'_i)\cap W_j^{(i-1)}$ have cardinality at least $(1/2+\del-\eps)|W_j^{(i-1)}|$ for each $j>i$. By the pigeonhole principle,
\[
|W_j^{(i)}|=|N(v_i)\cap N(v'_i)\cap W_j^{(i-1)}|\geq 2(\del-\eps)|W_j^{(i-1)}|
\]
for each $j>i$, as desired.

After $\ell$ steps, this process results in $k+2(\ell-k)=|S_1|+|S_2|$ vertices $v_1,\ldots,v_k,v_{k+1},v'_{k+1},\ldots,v_\ell,v'_\ell$ that form a clique in $G$. It follows that $|S_1|+|S_2|<t$, because $G$ is $K_t$-free.
\end{poc}

Let $R$ be the weighted graph on $[r]$ with vertex weights $w(i)=\frac 1r$ for all $i\in[r]$ and edge weights
\begin{equation*}
   w(i,j)=
    \begin{cases}
        \max\{d(V_i,V_j)-\del,0\},&\text{if $i\neq j$ and $(V_i,V_j)$ is $\eps$-regular}, \\ 
     0,&\text{if $i=j$ or $(V_i,V_j)$ is not $\eps$-regular}.
    \end{cases}
\end{equation*}
for all $i,j\in[r]$.
We observe that $R$ is $\C K_t$-free as a direct consequence of Claim \ref{reduction-UB-claim}.

To conclude the proof, we bound the $K_s$-density of $G$. We have
\[d_{K_s}(G)=\frac 1{N^s}\left(\sum_{a_1,\ldots,a_s\in[r+1]}\#\{(v_1,\ldots,v_s)\in V_{a_1}\times\cdots\times V_{a_s} \text{ that form a $K_s$ in $G$}\}\right).\]
If $a_1,\ldots,a_s$ are distinct elements of $[r]$ and each pair $(V_{a_i},V_{a_j})$ is $\eps$-regular, then we may simplify the summand using the graph-counting lemma. Indeed, \cref{lem:graph-counting} implies that the number of copies of $K_s$ in $V_{a_1}\times\cdots\times V_{a_s}$ is at most
\[
\left(\prod_{i=1}^s|V_{a_i}|\right)\left(\prod_{1\leq i<j\leq s}d(V_{a_i},V_{a_j})+\sqrt{\eps s^3}\right)\leq\frac{N^s}{r^s}\left(\prod_{1\leq i<j\leq s}d(V_{a_i},V_{a_j})+\sqrt{\eps s^3}\right)
\]
in this case.
It remains to bound the contribution from terms where some $a_i$ is $r+1$, the $a_i$ are not distinct, or some pair $(V_{a_i},V_{a_j})$ is not $\eps$-regular.

The terms where at least one index $a_i$ equals $r+1$ contribute at most
\[
\frac s{N^s}|V_{r+1}|N^{s-1}\leq \eps s
\]
to the sum.
The terms where $a_1,\ldots,a_s$ are not all distinct contribute at most
\[
\frac 1{N^s}\sum_{i=1}^r\binom s2|V_i|^2N^{s-2}\leq r\times\binom s2\times\frac 1{r^2}\leq \eps\binom s2
\]
because $r\geq 1/\eps$.
The terms where a pair $(V_{a_i},V_{a_j})$ is not $\eps$-regular contribute at most
\[
\frac 1{N^s}\sum_{\substack{1\leq i<j\leq r,\\(V_i,V_j)\text{ not}\\\text{$\eps$-regular}}}
	s(s-1)|V_i||V_j|N^{s-2}
<\frac{s(s-1)}{r^2}\times(\text{\# $\eps$-irregular pairs})
\leq \eps s(s-1).
\]

Combining these estimates, we have
\begin{align*}
d_{K_s}(G)&\leq \eps\left(s+3\binom s2\right)+r^{-s}\sum_{\substack{a_1,\ldots,a_s\in[r]\\\text{distinct}}}\left(\prod_{1\leq i<j\leq s}d(V_{a_i},V_{a_j})+\sqrt{\eps s^3}\right)
\\&\leq \eps\left(s+3\binom s2\right)+\sqrt{\eps s^3}+r^{-s}\sum_{\substack{a_1,\ldots,a_s\in[r]\\\text{distinct}}}\left(\prod_{1\leq i<j\leq s}d(V_{a_i},V_{a_j})\right).
\end{align*}
To compare this to $d_{K_s}(R)$, we observe the following inequality. If real numbers $x_1,\ldots,x_k,y_1,\ldots,y_k\in[0,1]$ satisfy $x_i\leq y_i+\del$ for each $i\in [k]$ then
\[
    \prod_{i=1}^kx_i-\prod_{i=1}^ky_i
    =(x_1-y_1)x_2\cdots x_k+y_1(x_2-y_2)x_3\cdots x_k + \cdots + y_1\cdots y_{k-1}(x_k-y_k)\leq k\del.
\]
Applying this inequality with $k=\binom s2$ to the real numbers $d(V_{a_i},V_{a_j})$ and $w(a_i,a_j)$, it follows that
\begin{align*}
d_{K_s}(G)
&\leq \eps\left(s+3\binom s2\right)+\sqrt{\eps s^3}+r^{-s}\sum_{\substack{a_1,\ldots,a_s\in[r]\\\text{distinct}}}\left(\prod_{1\leq i<j\leq s}w(a_i,a_j)+\del \binom s2\right)
\\&\leq2\eps s^2+\sqrt{\eps s^3}+\del \binom s2+d_{K_s}(R).
\end{align*}
Because $\eps<(\del-\eps)^{t-1}<\del^2$, it follows that $d_{K_s}(G)<d_{K_s}(R)+4s^2\del$, as desired.
\end{proof}

\subsection{Lower bound}

The lower bound construction for \cref{thm:reduction-to-R} hinges on a construction of Bollob\'as and Erd\H os \cite{Bollobas1976ramseyturan} which achieves the tight lower bound $\varrho_2(K_4)=\frac{1}{4}$. We briefly describe this construction, following the notation used in \cite{fox2015cw}.

Fix $0<\eps<1$ and an integer $h\geq 16$, and set $\mu=\frac{\eps}{\sqrt h}$. Let $X$ and $Y$ be sets of points on the unit sphere $\sph^{h-1}\subset\mathbb R^h$. The Bollob\'as--Erd\H os graph $\BE(X,Y)$ is a graph on vertex set $X\cup Y$ constructed as follows.
\begin{itemize}
    \item[(a)] Join $\textbf{\textit{x}} \in X$ to $\textbf{\textit{y}} \in Y$ if $\|\textbf{\textit{x}}-\textbf{\textit{y}}\|<\sqrt{2}-\mu$.
    \item[(b)] Join $\textbf{\textit{x}},\textbf{\textit{x}}'\in X$ if $\|\textbf{\textit{x}}-\textbf{\textit{x}}'\|>2-\mu$. Similarly, join $\textbf{\textit{y}},\textbf{\textit{y}}'\in Y$ if $\|\textbf{\textit{y}}-\textbf{\textit{y}}'\|>2-\mu$.
\end{itemize}
Bollob\'as and Erd\H os showed that this graph is $K_4$-free and that, if $\eps$ and $h$ are tuned appropriately and $X$ and $Y$ are uniformly placed on $\sph^{h-1}$, it has independence number $o(|X|+|Y|)$ and edge density $\frac 14-o(1)$.
Fox, Loh and Zhao \cite{fox2015cw} analyzed this construction in further detail, providing precise quantitative results on the independence number and minimum degree.
We need some results from their work.

\begin{lemma}[{\cite{fox2015cw}}]\label{lem:BE}
Let $0<\eps<1$ and let $h\geq 16$ be an integer. Set $\mu=\eps/\sqrt h$.
\begin{enumerate}[label=(\arabic*)]
\item If points $\textbf{\textit{x}},\textbf{\textit{y}}\in\sph^h$ are chosen independently and uniformly at random then $\Pr\left[\|\textbf{\textit{x}}-\textbf{\textit{y}}\|<\sqrt 2-\mu\right]\geq\frac 12-\sqrt 2\eps$. \label{BE-density}
\end{enumerate}
Now, fix $X,Y\subseteq\sph^{h-1}$, and let $G=\BE(X,Y)$ be the graph defined above with parameters $\eps$ and $h$.
\begin{enumerate}[label=(\arabic*)]
\setcounter{enumi}{1}
	\item The induced subgraphs $G[X]$ and $G[Y]$ are each $K_3$-free.\label{BE-K3free}
	\item $G$ is $K_4$-free. \label{BE-K4free}
	\item If $n$ is sufficiently large in terms of $\eps, h$, there is a choice $X\subset\sph^h$ of size $n$ such that the induced subgraph $G[X]$ has independence number at most $e^{-\eps\sqrt h/4}n$. \label{BE-alpha}
\end{enumerate}
\end{lemma}

Using these preliminaries, we prove the lower bound of \cref{thm:reduction-to-R}. It follows from \cref{reduction-LB} by choosing parameters $(\eps,h)$ such that $\eps\to 0$ and $\eps\sqrt h\to\infty$.

\begin{theorem}\label{reduction-LB}
Suppose $R$ is a weighted graph that is $\C K_t$-free for some integer $t\geq 3$. Fix $\eps>0$ and an integer $h\geq 16$. For all sufficiently large $N$, there is a $K_t$-free graph $G$ on $N$ vertices with independence number $\al(G)\leq 3e^{-\eps\sqrt h/4}N$ and $K_s$-density $d_{K_s}(G)\geq(1-2\sqrt 2s^2\eps)d_{K_s}(R)$.
\end{theorem}

\begin{proof}
Suppose that $V(R)=[r]$ for some integer $r$. Increasing each edge weight to the next multiple of $\frac 12$ preserves the $\C K_t$-freeness of $R$, so we may assume that all edge weights of $R$ are 0, $\frac 12$, or $1$. Set $\mu=\frac{\eps}{\sqrt h}$ as in the Bollob\'as--Erd\H os construction.

Choose integers $n_i\geq \lfloor w(i)N\rfloor$ such that $N=n_1+\cdots+n_r$. We construct an $N$-vertex graph $G$ on vertex set $V_1\cup\cdots\cup V_r$ as follows. Intuitively, $G[V_i,V_j]$ will be complete, empty, or a randomly rotated Bollob\'as--Erd\H os graph, depending on whether $w(i,j)$ is 1, 0, or $\frac 12$.

Suppose $N$ is sufficiently large. For each $i$, we may choose a set $V_i$ of $n_i$ points on $\sph^{h-1}$ satisfying \cref{lem:BE}\ref{BE-alpha}. Connect vertices in $\bigcup_iV_i$ as follows. Within each part $V_i$, add an edge between $\textbf{\textit{v}}_i, \textbf{\textit{v}}_i'\in V_i$ if $\|\textbf{\textit{v}}_i-\textbf{\textit{v}}_i'\|>2-\mu$. Let $G[V_i,V_j]$ be complete bipartite if $w(i,j)=1$ and empty if $w(i,j)=0$. If $w(i,j)=\frac 12$ for some $i<j$, then let $\rho_{ij}\in SO(h)$ be a rotation of $\sph^{h-1}$ chosen uniformly at random. Connect $\textbf{\textit{v}}_i\in V_i$ and $\textbf{\textit{v}}_j\in V_j$ if $\|\rho_{ij}\textbf{\textit{v}}_i-\textbf{\textit{v}}_j\|<\sqrt 2-\mu$.

Observe that each induced subgraph $G[V_i]$ is $K_3$-free with independence number $\al(G[V_i])\leq e^{-\eps\sqrt h/4}n_i$ by \cref{lem:BE}\ref{BE-K3free} and \ref{BE-alpha}. It follows that $\al(G)\leq e^{-\eps\sqrt h/4}N$.
Additionally, if $w(i,j)=\frac 12$, the induced subgraph $G[V_i\cup V_j]$ is the Bollob\'as--Erd\H os graph $\BE(\rho_{ij}(V_i),V_j)$, and is thus $K_4$-free by \cref{lem:BE}\ref{BE-K4free}.

Using these properties, we verify that $G$ is $K_t$-free. Suppose, for contradiction, that $G[W]$ is a clique, where $W$ is some set of $t$ vertices. Let $S_1=\{i\in[r]:|V_i\cap W|\geq 1\}$ and $S_2=\{i\in[r]:|V_i\cap W|\geq 2\}\subseteq S_1$. Because each induced subgraph $G[V_i]$ is $K_3$-free, it follows that $W$ has at most two points in $V_i$, and thus that $|S_1|+|S_2|=|W|=t$. For any distinct $i,j\in S_1$, there is an edge between $V_i$ and $V_j$, and it follows that $w(i,j)>0$. Additionally, for any distinct $i,j\in S_2$, there is a $K_4$ in $G[V_i\cup V_j]$. Because the Bollob\'as--Erd\H os graph is $K_4$-free, this implies that $w(i,j)>\frac 12$. We conclude that $(S_1,S_2)$ form a weighted $t$-clique in $R$, which is a contradiction. It follows that $G$ is $K_t$-free.

Lastly, we verify that $G$ has large $K_s$-density in expectation, using the independence of the random rotations $\rho_{ij}$. By \cref{lem:BE}\ref{BE-density}, if $w(i,j)=\frac 12$ then the expected edge density between $V_i$ and $V_j$ is at least $\frac 12-2\sqrt 2\eps$. Thus, if $N$ is sufficiently large, we have
\begin{align*}
\E[d_{K_s}(G)]&\geq\sum_{\substack{i_1,\ldots,i_s\in[r]\\\text{distinct}}}
\left(\prod_{j=1}^s\frac{n_{i_j}}N\right)\left(\prod_{1\leq j<k\leq s}
\begin{cases}0&w(i_j,i_k)=0\\
1/2-2\sqrt 2\eps&w(i_j,i_k)=1/2\\
1&w(i_j,i_k)= 1
\end{cases}\right)
\\
&\geq\sum_{\substack{i_1,\ldots,i_s\in[r]\\\text{distinct}}}\left(\prod_{j=1}^s(1-\eps)w(i_j)\right)\left(\prod_{1\leq j<k\leq s}(1-4\sqrt 2\eps)w(i_j,i_k)\right)
\\
&=(1-\eps)^s(1-4\sqrt 2\eps)^{s(s-1)/2}d_{K_s}(R)
\geq (1-2\sqrt 2s^2\eps)d_{K_s}(R).
\end{align*}
We conclude that there is a choice of the rotations $\rho_{ij}$ such that $d_{K_s}(G)\geq (1-2\sqrt 2 s^2\eps)d_{K_s}(R)$.
\end{proof}

\section{Understanding the extremal weighted graphs}\label{sec:extremal-weighted}
By Theorem~\ref{thm:reduction-to-R}, $\varrho_{s}(K_t)=\pi_{s}(\C K_t)$ for all $3\leq s\leq t-2$.
In this section, we show that the supremum $\pi_{s}(\C K_t)$ is attained by a weighted graph on at most $t-1$ vertices, and characterize the structure of a minimum-size extremal weighted graph more precisely. Our results are summarized as follows; together with Theorem~\ref{thm:reduction-to-R}, they imply Theorem~\ref{main 1}.

\begin{theorem}\label{thm:weight:main}
Fix integers $s,t$ satisfying $3\leq s\leq t-2$. There is an extremal $\C K_t$-free weighted graph $R$ achieving $K_s$-density $\pi_{s}(\C K_t)$ and satisfying the following properties.

\stepcounter{propcounter}
\begin{enumerate}[label = \rm({\bfseries \Alph{propcounter}\arabic{enumi}})]
\item\label{sup1} For any distinct $v,v'\in V(R)$, we have $w(v,v')\in\{\frac 12,1\}$.
\item\label{sup2} There is a partition $V(R)=B_1\cup\cdots\cup B_a$ into nonempty parts such that vertices inside the same part $B_i$ have the same weight, and an edge has weight $1/2$ if it lies within some $B_i$ and weight 1 otherwise. Moreover, setting $b=\sum_{i\in [a]}|B_i|=|V(R)|$, we have $b\geq s$ and $a+b= t-1$.
\item\label{sup3} For any $i$ and $j$, the cardinalities $|B_i|$ and $|B_j|$ differ by at most 1.
\item\label{sup4} If $|B_i|\geq|B_j|$ for any (possibly equal) $i,j\in[a]$, then $w(v_i)\leq w(v_j)$ for any $v_i\in B_i$ and $v_j\in B_j$. In particular, if $|B_i|=|B_j|$ then all vertices in $B_i$ and $B_j$ have the same weight.
\item\label{sup5} Either $a=1$ and $|B_1|=s$ or $a\geq 2$ and $|B_i| \le s-1$ for each $i\in [a]$.
\end{enumerate}
Moreover, all extremal $\C K_t$-free weighted graphs with minimum order satisfy~\ref{sup1}--\ref{sup5}.
\end{theorem}

Let $t$ and $s$ be integers such that $3\leq s\leq t-2$. We shall show in the following subsections that there exists a $\C K_t$-free weighted graph $R$ achieving $K_s$-density $\pi_{s}(\C K_t)$ with $|V(R)|$ minimized and satisfying~\ref{sup1}--\ref{sup5}.
Note that the weighted graph $R$ might not be unique.

Before beginning the proof, we introduce some notation which will be used in this section.
Let $R$ be a weighted graph and let $H$ be a graph on $[s]$. For a vertex set $S\subseteq V(R)$, the density of copies of $H$ containing $S$ is denoted by 
$$ d_H(R,S)=\sum_{\substack{\sigma: [s]\rightarrow  V(R) \\ S\subseteq \{ \sigma(1),\cdots,\sigma(s)\}}}\left(\prod_{i=1}^sw(\sigma(i))\right)\left(\prod_{ij\in E(H)}w(\sigma(i),\sigma(j))\right).$$
Similarly, write
\begin{align*}
d_H(R[S])&=\sum_{\substack{\sigma: [s]\rightarrow S}}\left(\prod_{i=1}^sw(\sigma(i))\right)\left(\prod_{ij\in E(H)}w(\sigma(i),\sigma(j))\right)
\quad\text{ and}
\\d_H(R-S)&=\sum_{\substack{\sigma: [s]\rightarrow  {V(R)\setminus S}}}\left(\prod_{i=1}^sw(\sigma(i))\right)\left(\prod_{ij\in E(H)}w(\sigma(i),\sigma(j))\right)
\end{align*}
for the density of copies of $H$ within $S$ and avoiding $S$, respectively. For convenience, we write $d_H(R,v)$ instead of $d_H(R,\{v\})$ and $d_H(R-v)$ instead of $d_H(R-\{v\})$. Additionally, when $H=K_0$, we define all densities to be 1.

We shall also require the following well-known inequality regarding symmetric functions. This is a special case of Maclaurin's inequality.

\begin{lemma}\label{maclaurin}
Let $x_1,\ldots,x_n$ be positive real numbers and let $x=(x_1+\cdots+x_n)/n$ be their average. For any integer $1\leq k\leq n$, we have
\[\sum_{1\leq i_1<\cdots<i_k\leq n}x_{i_1}x_{i_2}\cdots x_{i_k}\leq \binom nk x^k,\]
with equality if and only if all the $x_i$ are equal.
\end{lemma}

\subsection{Proof of~\ref{sup1}}

For each integer $n\geq s$, let $R_n$ be an $n$-vertex $\C K_t$-free weighted graph of maximum $K_s$-density. Such a weighted graph $R_n$ exists because the space of $n$-vertex weighted graphs, which is parametrized by possible choices of the vertex and edge weights, is compact.

We claim that $d_{K_s}(R_{n-1})\geq d_{K_s}(R_n)$ if $R_n$ contains an edge of weight 0. Suppose that $w(v_1,v_2) = 0$ for two distinct vertices $v_1,v_2\in V(R_n)$. For $i\in[2]$, let $R'_i$ be the $(n-1)$-vertex weighted graph obtained from $R_n$ by deleting $v_{3-i}$ and increasing the weight of $v_i$ to $w(v_1)+w(v_2)$. Clearly, both $R'_1$ and $R'_2$ are $\C K_t$-free. Moreover, writing $\alpha_i= \frac{w(v_i)}{w(v_1)+w(v_2)}$ for $i\in[2]$, we have
\begin{align}\label{eq:symm1}
\al_1\cdot d_{K_s}(R'_1)+\al_2\cdot d_{K_s}(R'_2)&=
d_{K_s}(R_n-\{ v_1,v_2\} )+d_{K_s}(R_n-v_2,v_1)+d_{K_s}(R_n-v_1,v_2)\nonumber
\\&= d_{K_s}(R_n).
\end{align}
This implies $d_{K_s}(R_{n-1})\geq\max_{i\in[2]}d_{K_s}(R'_i)\geq d_{K_s}(R_n)$. In particular, we have $d_{K_s}(R_{n-1})\geq d_{K_s}(R_n)$ for all $n\geq t$, as any $\C K_t$-free weighted graph on at least $t$ vertices must contain an edge of weight 0.

It follows that $\pi_s(\C K_t)=\sup_{n\geq s}d_{K_s}(R_n)$ is attained by a weighted graph $R_n$ on at most $t$ vertices. Moreover, any minimal-order weighted graph attaining the $K_s$-density $\pi_s(\C K_t)$ must have strictly positive edge weights.

We conclude the proof of~\ref{sup1} by observing that if $R$ is a $\C K_t$-free weighted graph with maximum $K_s$-density and strictly positive edge weights, then all edge weights of $R$ must be either $\frac 12$ or 1, as increasing any edge weight to the next multiple of $\frac 12$ will preserve the $\C K_t$-freeness of $R$ while increasing its $K_s$-density.

In the remaining subsections, we will show that any extremal $\mathcal K_t$-free weighted graph satisfying~\ref{sup1} --- and in particular, all such graphs of minimum order --- also satisfy~\ref{sup2}--\ref{sup5}.

\subsection{Proof of~\ref{sup2}}

Let $R$ be an extremal $\mathcal K_t$-free weighted graph satisfying~\ref{sup1}. We observe that $R$ must have at least $s$ vertices, as $d_{K_s}(R)$ would be 0 if $|V(R)|<s$. Moreover, if $R$ has exactly $s$ vertices, say $v_1,\ldots,v_s$, then
\[
d_{K_s}(R)=s!\left(\prod_{i=1}^sw(v_i)\right)\left(\prod_{1\leq i<j\leq s}w(v_i,v_j)\right)
\]
is maximized when the vertex weights are equal and the number of edges of weight 1 is maximized subject to the constraint that $R_{>1/2}$ is $K_{t-s}$-free. The latter condition holds if and only if $R_{>1/2}$, viewed as an unweighted graph, is the Tur\'an graph $T_{t-s-1}(s)$. Equivalently, $V(R)$ must admit a partition $V(R)=B_1\cup\cdots B_a$ with $a\leq t-s-1$ such that edges have weight $1/2$ if they lie within some part $B_i$ and weight 1 otherwise.

We now show that $R$ admits such a partition if $|V(R)|=b\geq s+1$. It suffices to show that any two vertices $v_1,v_2\in V(R)$ with $w(v_1,v_2)=\frac{1}{2}$ must be identical, i.e., $w(v_1)=w(v_2)$ and $e(v_1,u)=e(v_2,u)$ for any third vertex $u\in V(R)$.
For $i\in [2]$, let $R_i$ be the graph obtained from $R$ by changing the edge weight of $(v_{3-i},u)$ to $w(v_i,u)$ for all $u\in V(R)\setminus \{v_1,v_2\}$, and changing the vertex weights of both $v_1$ and $v_2$ to $\frac{w(v_1)+w(v_2)}{2}$. We claim that $R_1$ and $R_2$ are $\C K_t$-free. Indeed, suppose that $R_1$ contains a weighted $t$-clique $(S_1,S_2)$ with $S_2\subseteq S_1\subseteq V(R_1)$ and $|S_1|+|S_2|=t$. Because $w(v_1,v_2)=\frac 12$ in $R_1$, the set $S_2$ cannot contain both $v_1$ and $v_2$; as these vertices are indistinguishable in $R_i$, we may assume that $S_2$ does not contain $v_2$. Hence, $R_1$ and $R$ have the same edge weights between vertices of $S_2$. Furthermore, all edge weights of $R$ (and in particular, all edge weights between vertices of $S_1$) are positive because $R$ satisfies~\ref{sup1}. It follows that $S_1$ and $S_2$ form a weighted $t$-clique in $R$, a contradiction. The proof that $R_2$ is $\C K_t$-free is analogous.

Write $\alpha_i= \frac{w(v_i)}{w(v_1)+w(v_2)}$ for $i\in[2]$ as in~\eqref{eq:symm1}. We see that
$$\sum_{i\in[2]}\alpha_i\cdot \Big(d_{K_s}(R_i)-d_{K_s}(R_i,\{v_1,v_2\}) \Big)=d_{K_s}(R)-d_{K_s}(R,\{v_1,v_2\}).$$
To compare $d_{K_s}(R,\{v_1,v_2\})$ and $d_{K_s}(R_i,\{v_1,v_2\})$, let $S=\{v_1,\ldots,v_s\}\subseteq V(R)$ be any set of $s$ vertices containing $v_1$ and $v_2$. We observe that
\begin{align*}
    d_{K_s}(R,S)&=s!\cdot w(v_1)w(v_2)w(v_1,v_2)\left(\prod_{i=3}^{s}w(v_i)\right)
    \left(\prod_{i=3}^{s}w(v_1,v_i)\right)\left(\prod_{i=3}^{s}w(v_2,v_i)\right)
    \left(\prod_{s\ge i>j\ge3 }w(v_i,v_j)\right)\\
    &=\frac{s!}{2} w(v_1)w(v_2)W_1W_2W_3,
\end{align*}
where $W_1:=\prod_{i=3}^{s}w(v_1,v_i)$, $W_2:=\prod_{i=3}^{s}w(v_2,v_i)$, and $W_3:= \big(\prod_{i=3}^{s}w(i)\big)\big(\prod_{s\ge i>j\ge3 }w(v_i,v_j)\big)$. Furthermore, by the AM-GM inequality, we have
\[
\sum_{i\in[2]}\alpha_i\cdot d_{K_s}(R_i,S)
=\sum_{i\in[2]}\alpha_i\cdot\frac{s!}2\left(\frac{w(v_1)+w(v_2)}2\right)^2W_i^2W_3
\geq\frac{s!}2w(v_1)w(v_2)W_1W_2W_3
= d_{K_s}(R,S),
\]
with equality if and only if $w(v_1)=w(v_2)$ and $W_1=W_2$.

Summing over all such sets $S$, we have
\[
\sum_{i\in[2]}\alpha_i\cdot d_{K_s}(R_i,\{v_1,v_2\})
=\sum_{i\in[2]}\sum_{\substack{S\supseteq\{v_1,v_2\},\\|S|=s}}\alpha_i\cdot d_{K_s}(R_i,S)
\geq\sum_{\substack{S\supseteq\{v_1,v_2\},\\|S|=s}}d_{K_S}(R,S)
=d_{K_s}(R,\{v_1,v_2\}).
\]
Moreover, equality holds if and only if $w(v_1)=w(v_2)$ and $W_1=W_2$ for all sets $S$.
We claim that this condition implies that $w(v_1,u)=w(v_2,u)$ for each $u\in V(R)-\{v_1,v_2\}$. Indeed, because all edge weights are either $1/2$ or $1$, it follows that
\[
\sum_{i=3}^sw(v_1,v_i)=\sum_{i=3}^sw(v_2,v_i)
\]
for any $s-2$ distinct vertices $v_3,\ldots,v_s\in V(R)-\{v_1,v_2\}$. Letting $\mathbf w^{(1)},\mathbf w^{(2)}\in\mathbb \{\frac 12,1\}^{b-2}$ be vectors defined as $\mathbf w^{(i)}_u=w(v_i,u)$ for $u\in V(R)-\{v_1,v_2\}$, this yields a linear relation $\mathbf v\cdot\mathbf w^{(1)}=\mathbf v\cdot\mathbf w^{(2)}$, where $\mathbf v\in\mathbb R^{b-2}$ is the indicator vector of $\{v_3,\ldots,v_s\}$. When $b>s$, the vectors $\mathbf v$ span $\mathbb R^{b-2}$, and it follows that $\mathbf w^{(1)}=\mathbf w^{(2)}$.

We conclude that $\alpha_1\cdot d_{K_s}(R_1)+\alpha_2\cdot d_{K_s}(R_2)\geq d_{K_s}(R)$, with equality only if $w(v_1)=w(v_2)$ and $w(v_1,u)=w(v_2,u)$ for any third vertex $u$. Because $R$ is extremal, it follows that any $v_1,v_2\in V(R)$ with $w(v_1,v_2)=1/2$ must satisfy these conditions. This implies that $V(R)$ may be partitioned into $B_1\cup\cdots\cup B_a$ such that vertices within each part have the same weights, and edges have weight $1/2$ if they lie within some part $B_i$ and weight $1$ otherwise.

Lastly, we show that if $R$ is extremal and $V(R)$ admits such a partition $B_1\cup\cdots\cup B_a$ then $a+b=t-1$, where $b=|V(R)|\geq s$. If $a+b\geq t$ then we may form a weighted $t$-clique $(S_1,S_2)$ by setting $S_1=V(R)$ and letting $S_2$ contain one vertex from each of $B_1,\ldots,B_{t-b}$. If $a+b\leq t-2$ then we claim that $R$ is not extremal. Choose a vertex $v\in B_1$ and let $R'$ be the weighted graph obtained from $R$ by replacing $v$ with two vertices $v_1,v_2$ of weight $\frac{w(v)}2$, setting $w(v_1,v_2)=1/2$ and $w(v_i,u)=w(v,u)$ for $u\in V(R)-\{v\}$ and $i\in[2]$. We note that $R'$ is $\mathcal K_t$-free: $R'_{>1/2}$ is $K_{a+1}$-free, so for any weighted clique $(S_1,S_2)$ in $R'$, we have $|S_1|+|S_2|\leq |V(R')|+a=b+1+a\leq t-1$. Moreover, it is clear that $d_{K_s}(R')>d_{K_s}(R)$, contradicting the extremality of $R$. It follows that $a+b=t-1$, as desired.

\subsection{Proof of~\ref{sup3} and~\ref{sup4} for $a=2$}

We first prove~\ref{sup3} and~\ref{sup4} for weighted graphs $R$ satisfying~\ref{sup2} with $a=2$ parts. For convenience, we introduce the following notation. Given positive integers $P,Q$ and real numbers $p,q\in(0,1)$ satisfying $pP+qQ=1$, let $R(p,P;q,Q)$ denote the $(P+Q)$-vertex weighted graph satisfying~\ref{sup2} with parameters $a=2$, $|B_1|=P$, and $|B_2|=Q$, such that $w(v_1)=p$ and $w(v_2)=q$ for any vertex $v_1\in B_1$ or $v_2\in B_2$.

In \cref{P3-1,P3-2,P3-3} below, we show that if $R=R(p,P;q,Q)$ does not satisfy~\ref{sup3} or~\ref{sup4} then there is another weighted graph $R'$ on $P+Q$ vertices such that $R'_{>\frac 12}$ is also bipartite and $d_{K_m}(R')>d_{K_m}(R)$ for all $m$ in the range $2\leq m\leq P+Q$. This is a slightly stronger statement than necessary to handle the $a=2$ case, but it will prove necessary when we consider $a\geq 3$ in the next subsection.

\begin{lemma}\label{P3-1}
Let $P,Q$ be positive integers and $p,q\in(0,1)$ real numbers such that $pP+qQ=1$. If $P\geq Q+1$ and $(P-1)p>Qq$, then there exists a weighted graph $R'$ with $P+Q$ vertices such that $R'_{>\frac 12}$ is bipartite and $d_{K_m}(R(p,P;q,Q))< d_{K_m}(R')$ for all $2\leq m\leq P+Q$.
\end{lemma}

\begin{proof}
Set $R = R(p,P;q,Q)$ and let $u$ be a vertex in $R$ with weight $p$. Let $R'$ be the graph obtained from $R$ by changing the weights of all edges incident to $u$: if $(u,v)$ has edge weight $w\in\{\frac 12,1\}$ in $R$, then we assign it the weight $\frac 32-w$ in $R'$. It is clear that $R'_{>\frac 12}$ is a complete bipartite graph with parts of size $P-1$ and $Q+1$.

Fix $m$ with $2\leq m\leq P+Q$. We verify that $d_{K_m}(R,u)<d_{K_m}(R',u)$. We have that
\begin{align*}
    &d_{K_m}(R,u) = m!\sum_{x+y=m-1}\binom{P-1}{x}\binom{Q}{y}p^{x+1} q^y \left(\frac{1}{2}\right)^{\binom{x}{2}+\binom{y}{2}+x},\\
    &d_{K_m}(R',u) =m! \sum_{x+y=m-1}\binom{P-1}{x}\binom{Q}{y}p^{x+1} q^y \left(\frac{1}{2}\right)^{\binom{x}{2}+\binom{y}{2}+y}.
\end{align*}
Let $M(x,y) = \binom{P-1}{x}\binom{Q}{y}p^{x+1} q^y $. If $x\ge y$ then
\begin{align*}
    \frac{M(x,y)}{M(y,x)} = \frac{p^{x-y}}{q^{x-y}} \prod_{i=1}^{x-y}\frac{P-y-i}{Q+1-y-i}\ge \frac{p^{x-y}}{q^{x-y}}\left(\frac{P-1}{Q}\right)^{x-y}\ge 1,
\end{align*}
with equality only if $x=y$. This in turn implies that
\[M(x,y)\left(\frac{1}{2}\right)^{x}+M(y,x)\left(\frac{1}{2}\right)^{y} \le M(x,y)\left(\frac{1}{2}\right)^{y}+M(y,x)\left(\frac{1}{2}\right)^{x}\]
for any $x,y$, again with equality only if $x=y$. Thus, 
 \begin{align*}
    2d_{K_m}(R,u)&= m!\sum_{x+y=m-1}\left(\frac{1}{2}\right)^{\binom{x}{2}+\binom{y}{2}}\bigg(
    M(x,y)\left(\frac{1}{2}\right)^{x}+M(y,x)\left(\frac{1}{2}\right)^{y}\bigg)\\
    &< m!\sum_{x+y=m-1}\left(\frac{1}{2}\right)^{\binom{x}{2}+\binom{y}{2}}\bigg(
    M(x,y)\left(\frac{1}{2}\right)^{y}+M(y,x)\left(\frac{1}{2}\right)^{x}\bigg) = 2d_{K_m}(R',u).    
 \end{align*}
To conclude the proof, we observe that $d_{K_m}(R-u)=d_{K_m}(R'-u)$, and thus
\[d_{K_m}(R)=d_{K_m}(R,u)+d_{K_m}(R-u)<d_{K_m}(R',u)+d_{K_m}(R'-u)=d_{K_m}(R').\qedhere\]
\end{proof}

Next, we handle~\ref{sup3} in the case that $(P-1)p\leq Qq$. To help us bound $d_{K_m}(R(p,P;q,Q))$ in this case, we shall write it in terms of the following quantities. Given integers $m,r$ with $0\leq r\leq m/2$, define
\begin{align*}
    N_{m,r}(p,P;q,Q) &= r!\cdot\binom{P}{r}p^r \cdot r!\cdot\binom{Q}{r}q^r\sum_{\substack{x+y=m,\\x, y\ge r}} \binom{P-r}{x-r}p^{x-r}\binom{Q-r}{y-r}q^{y-r}\\
    &=\sum_{\substack{x+y=m,\\x, y\ge r}} \binom{P}{x}p^{x}\binom{Q}{y}q^{y}\frac{x!y!}{(x-r)!(y-r)!}.
\end{align*}
Intuitively, $N_{m,r}$ should be thought of as counting copies of $K_m$ in $R(p,P;q,Q)$ with $r$ labeled vertices in each part. We now show that $d_{K_m}(R(p,P;q,Q))$ is a positive linear combination of these quantities.

\begin{lemma}\label{P3-2a}
Fix a positive integer $m$. There exist constants $c_{r}>0$ for $0\leq r\leq \lfloor m/2\rfloor$ such that
\[
d_{K_m}(R(p,P;q,Q))=\sum_{r=0}^{\lfloor m/2\rfloor}c_{r}N_{m,r}(p,P;q,Q)
\]
for any $P,Q,p,q$.
\end{lemma}

\begin{proof}
Observe that each $N_{m,r}$ is a linear combination of the $\left\lfloor\frac m2\right\rfloor +1$ polynomials
\[\left\{\binom Px\binom Q{m-x}p^xq^{m-x}+\binom P{m-x}\binom Qxp^{m-x}q^x : 0\leq x\leq\frac m2\right\},\]
with nonzero coefficient if and only if $x\geq r$. Thus, $\{N_{m,r}:0\leq r\leq m/2\}$ is a basis for the space of all linear combinations of these polynomials, which includes the $K_m$-density
\[d_{K_m}(R(p,P;q,Q))=m!\cdot \sum_{x+y=m}\binom{P}{x}\binom{Q}{y}p^xq^y2^{xy-\binom{m}{2}}
= \frac{m!}{2^{\binom{m}{2}}}\cdot \sum_{x+y=m}\binom{P}{x}\binom{Q}{y}p^xq^y2^{xy}.
\]
It follows that there exist real numbers $c_{r}$ such that
\begin{align*}
  d_{K_m}(R(p,P;q,Q)) =\sum_{r=0}^{\lfloor m/2\rfloor} c_{r}N_{m,r}(p,P;q,Q).
\end{align*}
Moreover, by setting the coefficients of the $r$th term equal to each other, we conclude that
\begin{equation}\label{eq:Ntod}
   \frac{m!}{2^{\binom{m}{2}}} \cdot 2^{r(m-r)} = \sum_{i=0}^r\frac{r!(m-r)!}{(r-i)!(m-r-i)!}c_i
\end{equation}
for each $r\leq m/2$.

We show that the coefficients $c_{r}$ are positive by induction on $r$. Clearly, $c_0> 0$ by~\eqref{eq:Ntod} when $r=0$. Now, suppose $c_i> 0$ for all $i<r$. By~\eqref{eq:Ntod}, we have
\begin{align*}
  \sum_{i=0}^r\frac{r!(m-r)!}{(r-i)!(m-r-i)!}c_i &= 2^{r(m-r)}\cdot \frac{m!}{2^{\binom{m}{2}}} = \frac{m!}{2^{\binom{m}{2}}} 2^{m-2r+1}2^{(r-1)(m-r+1)}\\
  &=
  2^{m-2r+1}\sum_{i=0}^{r-1}\frac{(r-1)!(m-r+1)!}{(r-1-i)!(m-r+1-i)!}c_i.
\end{align*}
We claim that
\[2^{m-2r+1}\frac{(r-1)!(m-r+1)!}{(r-1-i)!(m-r+1-i)!}>\frac{r!(m-r)!}{(r-i)!(m-r-i)!}\]
for each $i<r$. Indeed, it suffices to show that
\[2^{m-2r+1} \frac{r-i}{m-r+1-i}\ge 1>\frac{r}{m-r+1},\]
which, recalling that $r\leq m/2$, follows from the inequalities $2^{m-2r+1}\geq m-2r+2$ and $\frac{r-i}{m-r+1-i}=\frac{r-i}{m-2r+1+(r-i)}\geq\frac{1}{m-2r+2}$.
Hence,
\begin{align*}
  \sum_{i=0}^r\frac{r!(m-r)!}{(r-i)!(m-r-i)!}c_i = 2^{m-2r+1}\sum_{i=0}^{r-1}\frac{(r-1)!(m-r+1)!}{(r-1-i)!(m-r+1-i)!}c_i
  >\sum_{i=0}^{r-1}\frac{r!(m-r)!}{(r-i)!(m-r-i)!}c_i,
\end{align*}
which implies $c_r>0$.
\end{proof}

Using \cref{P3-2a}, we handle~\ref{sup3} in the case that $(P-1)p\leq Qq$.

\begin{lemma}\label{P3-2}
Let $P,Q$ be positive integers with $P\geq Q+2$ and let $p,q\in(0,1)$ be real numbers such that $Pp+Qq=1$. Set $P'=P-1$, $Q'=Q+1$, and choose $p',q'\in (0,1)$ such that $P'p'=Pp$ and $Q'q'=Qq$. If $(P-1)p\le Qq$ then $d_{K_m}(R(p,P;q,Q))< d_{K_m}(R(p',P';q',Q'))$ for all $2\leq m\leq P+Q$.
\end{lemma}

\begin{proof}
Set $R=R(p,P;q,Q)$ and $R'=R(p',P';q',Q')$, and fix $m$ with $2\leq m\leq P+Q$. For convenience, we write $N_{m,r}(R)=N_{m,r}(p,P;q,Q)$ and $N_{m,r}(R')=N_{m,r}(p',P';q',Q')$.

By \cref{P3-2a}, it suffices to compare $N_{m,r}(R)$ and $ N_{m,r}(R')$. We begin with some preliminary inequalities. Let $p_-=\min\{p',q'\}$ and $p_+=\max\{p',q'\}$.

\begin{claim}\label{clm:pq} We have the following inequalities.
\begin{enumerate}[label=(\arabic*)]
	\item For any $0<r \le |Q|$, we have $(1-\frac{r}{P})(1-\frac{r}{Q})< (1-\frac{r}{P'})(1-\frac{r}{Q'})$.
	\item $p\leq p_-\leq p_+<q$.
	\item For any $0<r\leq|Q|$, we have $(P-r)p+(Q-r)q<(P'-r)p'+(Q'-r)q'$.
\end{enumerate}
\end{claim}

\begin{poc}
(1) Because $P\geq Q+2$, we have
\begin{eqnarray*}
    \Big(1-\frac{r}{P'}\Big)\Big(1-\frac{r}{Q'}\Big)-\Big(1-\frac{r}{P}\Big)\Big(1-\frac{r}{Q}\Big)&=&\Big(1+\frac{r^2-(P+Q)r}{(P-1)(Q+1)}\Big)-\Big(1+\frac{r^2-(P+Q)r}{PQ}\Big)\\
    &=&\Big(r^2-(P+Q)r\Big)\Big(\frac{1}{PQ+P-1-Q}-\frac{1}{PQ}\Big)> 0.
\end{eqnarray*}

(2) Using the relation $(P-1)p\leq Qq$, it follows that $p<\frac{P}{P-1}p=p'\leq\frac{PQ}{(P-1)^2}q<q$ and that $q>\frac{Qq}{Q+1}=q'\geq\frac{(P-1)p}{Q+1}\geq p$.

(3) First, observe that
\[
p+q-p'-q'
=p+q-\frac{P}{P-1}p-\frac{Q}{Q+1}q
=\frac{q}{Q+1}-\frac{p}{P-1}>0
\]
because $p\leq Qq/(P-1)<q$. Thus,
\[(P'-r)p'+(Q'-r)q'-(P-r)p-(Q-r)q
=(P'p'+Q'q'-Pp-Qq)+r(p+q-p'-q')>0
\]
if $r>0$.
\end{poc}

We now compare $N_{m,r}(R)$ and $N_{m,r}(R')$.

\begin{claim}\label{clm:NmrR}
We have $N_{m,r}(R)\leq N_{m,r}(R')$ for any $0\leq r\leq\lfloor m/2\rfloor$. Moreover, the inequality is strict when $r>0$.	
\end{claim}

\begin{poc}
By \cref{clm:pq}(1),
\begin{align*}
r!\cdot\binom{P}{r}p^r \cdot r!\cdot\binom{Q}{r}q^r
&=(Pp)^r(Qq)^r\prod_{i=0}^{r-1}\left(1-\frac iP\right)\left(1-\frac iQ\right)
\\&\leq (P'p')^r(Q'q')^r\prod_{i=0}^{r-1}\left(1-\frac i{P'}\right)\left(1-\frac i{Q'}\right)
=r!\cdot\binom{P'}{r}(p')^r \cdot r!\cdot\binom{Q'}{r}(q')^r,
\end{align*}
with equality only if $r=0$.
It remains to show that
\begin{equation}\label{eq:symm}
	\sum_{\substack{x+y=m-2r\\x, y\ge 0}} \binom{P-r}{x}p^{x}\binom{Q-r}{y}q^{y}\leq
	\sum_{\substack{x+y=m-2r\\x, y\ge 0}} \binom{P'-r}{x}(p')^{x}\binom{Q'-r}{y}(q')^{y}.
\end{equation}
Set $X = (\underbrace{p,p ,\ldots, p }_{P-r}, \underbrace{q,q ,\ldots,q}_{Q-r})$  and $X' = (\underbrace{p',p' ,\ldots, p' }_{P-r-1}, \underbrace{q',q' ,\ldots,q'}_{Q-r+1})$. Letting $\sigma$ be the $(m-2r)$th elementary symmetric function
\[
\sigma(y_1,\ldots,y_{P+Q-2r}):=\sum_{1\leq i_1<\cdots<i_{m-2r}\leq P+Q-2r}y_{i_1}\cdots y_{i_{m-2r}},
\]
we may rewrite~\eqref{eq:symm} as $\sigma(X)\leq\sigma(X')$.

Let $Y$ be the $(P+Q-2r)$-tuple obtained via the following process. Set $Y=X$ initially and repeat the following transformation, which will never decrease $\sigma(Y)$.
\begin{enumerate}
\item[$(*)$]  If there exist indices $i,j$ such that $Y_i<p_-$ and $Y_j>p_+$, set $\eps=\min\{p_--Y_i,~Y_j-p_+\}$, and replace $Y_i$ and $Y_j$ with $Y_i+\eps$ and $Y_j-\eps$, respectively.
\end{enumerate}
\noindent Each iteration increases the number of coordinates equal to $p_-$ or $p_+$, so the process will terminate in at most $P+Q-2r$ steps. Moreover, recalling that $p\leq p_-\leq p_+<q$, we observe that the final tuple $Y$ either takes the form
\begin{align*}
Y&=(Y_1,\ldots,Y_{P-r},p_+,\ldots,p_+)\text{ with }Y_i\leq p_-\text{ for all $i\leq P-r$, or}\tag{Case 1}
\\Y&=(p_-,\ldots,p_-,Y_{P-r+1},\ldots,Y_{P+Q-2r})\text{ with }Y_i\geq p_+\text{ for all $i>P-r$.}\tag{Case 2}
\end{align*}
Let $X'' = (p_-,\ldots,p_-,p_+,\ldots,p_+)$ be the result of sorting $X'$ in increasing order, and let $k\in\{P-r-1,Q-r+1\}$ be the number of occurrences of $p_-$. In case 1, we have $Y_i\leq X''_i$ for each $i$, implying $\sigma(Y)\leq \sigma(X'')$. In case 2, let $y$ be the average value of $\{Y_{k+1},\ldots,Y_{P+Q-2r}\}$ and let $Y'=(p_-,\ldots,p_-,y,\ldots,y)$ be the result of replacing all but the first $k$ terms of $Y$ with $y$. By \cref{maclaurin}, $\sigma(Y)\leq\sigma(Y')$. Moreover, $\mathrm{sum}(Y')=\mathrm{sum}(X)<\mathrm{sum}(X'')$ by \cref{clm:pq}(3). Because $X''$ is obtained from $Y'$ by replacing each $y$ with $p_+$, it follows that $y<p_+$ and $\sigma(Y')<\sigma(X'')$. Thus, we have $\sigma(X)\leq\sigma(Y)\leq \sigma(X'')=\sigma(X')$ in both cases, completing the proof that $N_{m,r}(R)\leq N_{m,r}(R')$.
\end{poc}

Combining \cref{P3-2a} and \cref{clm:NmrR}, we have that
\[
d_{K_m}(R)=\sum_{r=0}^{\lfloor m/2\rfloor}c_rN_{m,r}(R)<\sum_{r=0}^{\lfloor m/2\rfloor}c_rN_{m,r}(R')=d_{K_m}(R').\qedhere
\]
\end{proof}

Lastly, we turn our attention to~\ref{sup4}. If $P>Q$, this is an immediate consequence of \cref{P3-1}; it remains to handle the $P=Q$ case.

\begin{lemma}\label{P3-3}
Let $P$ be a positive integer and $p,q\in(0,1)$ real numbers such that $pP+qP=1$. If $p\neq q$, then there exists a weighted graph $R'$ with $2P$ vertices such that $R'_{>\frac 12}$ is bipartite and $d_{K_m}(R(p,P;q,P))< d_{K_m}(R')$ for all $2\leq m\leq 2P$.
\end{lemma}

\begin{proof}
Set $R=R(p,P;q,P)$ and let $u$ and $v$ be vertices in $R$ with weights $p$ and $q$, respectively. Let $R'$ be the graph obtained from $R$ by setting the weights of both $u$ and $v$ to $\frac{p+q}2=\frac 1{2P}$.

Observe that $d_{K_m}(R-\{u,v\})=d_{K_m}(R'-\{u,v\})$. Set
\[f(x,y)=m!\binom{P-1}{x}\binom{P-1}{y}\left(\frac{1}{2}\right)^{\binom{x}{2}+\binom{y}{2}}.\]
We have that
\begin{align*}
&d_{K_m}((R'-v),u)+d_{K_m}((R'-u),v)-d_{K_m}((R-v),u)-d_{K_m}((R-u),v)
\\&=\sum_{x+y=m-1}f(x,y)p^xq^y
	\left[\frac{p+q}{2}\left(\left(\frac{1}{2}\right)^x+\left(\frac{1}{2}\right)^y\right)-\left(p\left(\frac{1}{2}\right)^x+q\left(\frac{1}{2}\right)^y\right)\right]
\\&=\sum_{x+y=m-1}f(x,y)p^xq^y\left(\frac{q-p}2\right)
	\left(\left(\frac{1}{2}\right)^x-\left(\frac{1}{2}\right)^y\right)
\\&=\sum_{\substack{x+y=m-1,\\x<y}}f(x,y)
	p^xq^x\left(q^{y-x}-p^{y-x}\right)
	\left(\frac{q-p}2\right)
	\left(\left(\frac{1}{2}\right)^x-\left(\frac{1}{2}\right)^y\right)>0
\end{align*}
because $q-p$ and $q^{y-x}-p^{y-x}$ both have the same sign.
Additionally,
\begin{align*}
d_{K_m}(R',\{u,v\})-d_{K_m}(R,\{u,v\})
=\sum_{x+y=m-2}f(x,y)p^xq^y\left(\frac 12\right)^{x+y}
	\left[\left(\frac{p+q}2\right)^2-pq\right]
	>0.
\end{align*}
Combining the preceding inequalities, we conclude that
\begin{align*}
d_{K_m}(R) &= d_{K_m}(R-\{u,v\})+d_{K_m}((R-v),u)+d_{K_m}((R-u),v)+d_{K_m}(R,\{u,v\})
\\&< d_{K_m}(R'-\{u,v\})+d_{K_m}((R'-v),u)+d_{K_m}((R'-u),v)+d_{K_m}(R',\{u,v\})
=d_{K_m}(R')
\end{align*}
for any integer $2\leq m\leq 2P$.
\end{proof}

\subsection{Proof of~\ref{sup3} and~\ref{sup4} for all $a$}

Using the results of the prior subsection, we prove~\ref{sup3} and~\ref{sup4} for all $a$. Suppose $R$ is an extremal $\mathcal K_t$-free weighted graph satisfying~\ref{sup1} and~\ref{sup2} with parts $B_1,\ldots,B_a$.
Given indices $i\neq j$, we may regard $R[B_i\cup B_j]$ as a scaled-down version of some weighted graph $R_0=R(p,|B_i|;q,|B_j|)$, where the weights of vertices in $B_i$ and $B_j$ are $\alpha p$ and $\alpha q$ respectively, for some $\alpha\leq 1$.

Without loss of generality, suppose $|B_i|\geq |B_j|$. We claim that $|B_i|\leq|B_j|+1$ and that $p\leq q$. Indeed, if $|B_i|\geq|B_j|+2$ then \cref{P3-1} or \cref{P3-2} yields a weighted graph $R'_0$ on $|B_i|+|B_j|$ vertices such that $(R'_0)_{>\frac 12}$ is bipartite and $d_{K_m}(R'_0)> d_{K_m}(R_0)$ for all $2\leq m\leq |B_i|+|B_j|$. Note that $d_{K_m}(R_0)=d_{K_m}(R'_0)$ for all other $m$: this value is 1 if $m=1$ and 0 if $m>|B_i|+|B_j|$. If $p>q$, then \cref{P3-1} (if $|B_i|>|B_j|$) or \cref{P3-3} (if $|B_i|=|B_j|$) provides a weighted graph $R'_0$ with the same properties.

Let $R'$ be the weighted graph obtained from $R$ by replacing $R[B_i\cup B_j]$ with a scaled-down copy of $R'_0$. That is, the vertex weights of $R'[B_i\cup B_j]$ are those of $R'_0$ multiplied by $\alpha$, and the edge weights of $R'[B_i\cup B_j]$ are exactly those of $R'_0$.
We verify that $R'$ is $\mathcal K_t$-free. Suppose $(S_1,S_2)$ is a weighted clique in $R'$. Because $\left(R'_0\right)_{>1/2}$ is bipartite, $S_2$ contains at most two vertices of $B_i\cup B_j$; additionally, $S_2$ contains at most one vertex from each other part $B_k$. Hence $|S_1|+|S_2|\leq|V(R)|+a=b+a=t-1$, so $R'$ is $\mathcal K_t$-free.
However, because
\[d_{K_m}(R[B_i\cup B_j])=\alpha^m d_{K_m}(R_0)
\quad\text{and}\quad d_{K_m}(R'[B_i\cup B_j])=\alpha^m d_{K_m}(R'_0),\]
we have that
\begin{align*}
d_{K_s}(R)
&=\sum_{m=0}^s\binom sm \al^md_{K_m}(R_0) \cdot d_{K_{s-m}}(R-(B_i\cup B_j))
\\&<\sum_{m=0}^s\binom sm \al^md_{K_m}(R'_0) \cdot d_{K_{s-m}}(R-(B_i\cup B_j))
=d_{K_s}(R').
\end{align*}
Thus, if the parts $(B_i,B_j)$ do not satisfy~\ref{sup3} or~\ref{sup4} then $d_{K_s}(R)<d_{K_s}(R')$, contradicting the extremality of $R$.

\subsection{Proof of~\ref{sup5}}

Suppose that $R$ is an extremal $\mathcal K_t$-free weighted graph satisfying~\ref{sup1} and thus~\ref{sup2}--\ref{sup4} with parts $B_1,\ldots,B_a$. Without loss of generality, suppose $B_1$ has maximal cardinality among all parts $B_i$. We assume that $r=|B_1|$ satisfies $r\geq s+1$ (if $a=1$) or $r\geq s$ (if $a\geq 2$) and derive a contradiction.

Let $p$ be the weight of each vertex in $B_1$; by~\ref{sup4}, it follows that all vertices of $R$ have weight at least $p$. Let $R'$ be the graph obtained from $R$ by replacing two vertices $u,v\in B_1$ with a new vertex $v'$ of weight $2p$ and setting $w(v',u')=1$ for any $u'\in V(R')\setminus\{v'\}$. We observe that $R'$ is $\mathcal K_t$-free. Indeed, if $(S_1,S_2)$ is a weighted clique configuration in $R'$ then $S_2$ may contain at most one vertex from each of the sets $\{v'\},B_1-\{u,v\},B_2,\ldots,B_a$. It follows that $|S_1|+|S_2|\leq|V(R')|+(a+1)=|V(R)|+a$, which is $t-1$ by~\ref{sup2}. To conclude the proof of ~\ref{sup5}, we show that $d_{K_s}(R)<d_{K_s}(R')$ if $R$ does not satisfy~\ref{sup5}, contradicting the extremality of $R$.

Set $B'=(B_1-\{u,v\})\cup\{v'\}\subseteq V(R')$ and for each $0\leq m\leq s$ set
\begin{align*}
f_m&=\frac{d_{K_m}(R[B_1])}{m!},
\quad g_m=\frac{d_{K_m}(R'[B'])}{m!},
\quad\text{and}\quad
h_m=\frac{d_{K_m}(R-B_1)}{m!}=\frac{d_{K_m}(R'-B'))}{m!}.
\end{align*}
Recalling that $|B_1|=r\geq s \geq 3$, we have
\begin{align*}
f_m&=\binom rmp^m\left(\frac 12\right)^{\binom m2},\text{ and}
\\g_m&=\binom{r-2}mp^m\left(\frac 12\right)^{\binom m2}+\binom{r-2}{m-1}(2p)p^{m-1}\left(\frac 12\right)^{\binom{m-1}2}
\\&=\left[\binom{r-2}m+2^m\binom{r-2}{m-1}\right]p^m\left(\frac 12\right)^{\binom m2}.
\end{align*}
For any $1\leq m\leq r-1$, observe that
\begin{align*}
\binom{r-2}m+2^m\binom{r-2}{m-1}-\binom rm
&=\binom{r-1}m+\left(2^m-1\right)\binom{r-2}{m-1}-\left(\binom{r-1}{m}+\binom {r-1}{m-1}\right)
\\&=\left(2^m-1-\frac{r-1}{r-m}\right)\binom{r-2}{m-1}
\geq\left(2^m-1-m\right)\binom{r-2}{m-1}\geq 0,
\end{align*}
with equality if and only if $m=1$. Recalling that $f_0=g_0=1$, we conclude that $f_m\leq g_m$ for all $0\leq m\leq r-1$ with equality if and only if $m<2$. If $r\geq s+1$, then this inequality holds for all $2\leq m\leq s\leq r-1$, and we conclude that
\begin{align*}
d_{K_s}(R)&=\sum_{m=0}^s\binom smd_{K_m}(R[B_1])d_{K_{s-m}}(R-B_1)
=\sum_{m=0}^ss!f_mh_{s-m}
\\&<\sum_{m=0}^ss!g_mh_{s-m}
=\sum_{m=0}^s\binom sm d_{K_m}(R'[B'])d_{K_{s-m}}(R'-B')=d_{K_s}(R').
\end{align*}
Now, suppose $r=s$ and $a\geq 2$. We observe that $h_1=\sum_{v\in V(R)-B_1}w(v)\geq|V(R)-B_1|p\geq p$, so
\begin{align*}
   f_{s-1}h_1+f_sh_0 &= sp^{s-1}\left(\frac{1}{2}\right)^{\binom{s-1}{2}}h_1+p^s\left(\frac{1}{2}\right)^{\binom{s}{2}} 
   \le \left(s\cdot \left(\frac{1}{2}\right)^{s-2}+\left(\frac{1}{2}\right)^{2s-3}\right) p^{s-1}\left(\frac{1}{2}\right)^{\binom{s-2}{2}}h_1 \\
   & < (2p)p^{s-2}\left(\frac{1}{2}\right)^{\binom{s-2}{2}}h_1 = g_{s-1}h_1\leq g_{s-1}h_1+g_sh_0.
\end{align*}
Once again, we conclude
\[
d_{K_s}(R)=\sum_{m=0}^ss!f_mh_{s-m}<\sum_{m=0}^ss!g_mh_{s-m}=d_{K_s}(R').
\]
Thus, if $R$ does not satisfy~\ref{sup5} then $d_{K_s}(R')>d_{K_s}(R)$, contradicting the extremality of $R$.
This completes the proof of Theorem~\ref{thm:weight:main}.

\section{Eventual periodicity and counterexamples to Conjecture~\ref{conj}}\label{section:5}

In this section, we prove \cref{main 2}, providing conditions under which \cref{conj} does and does not hold. Applying \cref{thm:reduction-to-R}, we may reduce this to a problem about weighted graphs.

Say a weighted graph $R$ \emph{admits a $(b,a)$-partition} if it has $b$ vertices and satisfies the five conditions \ref{sup1}--\ref{sup5} of \cref{thm:weight:main}, with $a$ being the number of parts in \ref{sup2}.
\cref{main 1} shows that $\varrho_s(K_t)$ is the maximum $K_s$-density of a weighted graph $R$ admitting a $(b,a)$-partition with $a=t-1-b$ parts for some $b<t$; such $R$ are inherently $\C K_t$-free. From \ref{sup2}, we have $b\geq s$; additionally, because $R$ cannot have fewer vertices than parts, we have $b\geq \lfloor t/2\rfloor$. \cref{conj} hypothesizes that the optimal density is attained when $b$ matches one of these lower bounds.

\begin{conj}[\cref{conj}, rephrased in terms of weighted graphs]\label{conj-wg}
	Fix integers $s,t$ with $3\leq s\leq t-2$. The maximum $K_s$-density of a weighted graph admitting a $(b,t-1-b)$-partition is attained when $b=\max\{s,\lfloor t/2\rfloor\}$.
\end{conj}

For $t\geq 2s$, \cref{conj-wg} hypothesizes that the extremal $\C K_t$-free weighted graphs follow an alternating pattern as depicted in \cref{tab:my_label}. For odd $t$, the conjectural extremal construction is the \emph{complete balanced weighted graph} $K_{\lfloor t/2\rfloor}^w$, which has $b=\lfloor t/2\rfloor$ vertices with  weight $1/b$ each and has all $\binom b2$ edge weights equal to 1. For even $t$, the conjectural extremal construction has $b=t/2$ vertices divided into one part of size 2 and $b-2$ parts of size 1. That is, all edges have weight $1$ except for one edge of weight $1/2$ within the part of size 2.

The proof of \cref{main 2} is divided into three subsections. In \cref{sec:case 1}, we show that \cref{conj-wg} holds if $t>s^2(s-1)/2 + s + 1$, implying that the extremal constructions do eventually follow the aforementioned alternating pattern. In \cref{sec:case 2}, we show that \cref{conj-wg} holds if $s=3,4$ or if $t=s+2$. Lastly, in \cref{sec:case 3}, we provide counterexamples to \cref{conj-wg} for $s=5$ as well as for any sufficiently large $s$.

\subsection{Eventual Periodicity}\label{sec:case 1}

We first show that \cref{conj-wg} holds when $t$ is sufficiently large as a function of $s$.

\begin{lemma}
Fix integers $s,t$ with $3\leq s\leq t-2$. Suppose $R$ is a weighted graph admitting a $(b,a)$-partition into $B_1\cup\cdots\cup B_a$ for some $a,b$ satisfying $a+b=t-1$.
\begin{enumerate}[label=(\arabic*)]
	\item Suppose $t$ is odd with $t>s^2(s-1)/2$. Set $r=(t-1)/2$ and let $K_r^w$ be the complete balanced weighted graph on $r$ vertices. We have $d_{K_s}(R)\leq d_{K_s}(K_r^w)$ with equality if and only if $R=K_r^w$.
	\item Suppose $t$ is even with $t>s^2(s-1)/2+s$. If $d_{K_s}(R)=\pi_s(\C K_t)$ then $a-1$ of the parts $B_i$ must have cardinality 1, and the last part must have cardinality 2.
\end{enumerate}
\end{lemma}

\begin{proof}
We first prove (1). Suppose $t>s^2(s-1)/2$ is odd. Let $R_0$ be the weighted graph obtained from $R$ by changing all edges weights of $1/2$ into 0 and let $R_1$ be the weighted graph obtained from $R$ by changing all edge weights of $1/2$ into 1. We claim that
\[d_{K_s}(R)\leq \frac{d_{K_s}(R_0)+d_{K_s}(R_1)}2.\]
Indeed, if vertices $v_1,\ldots,v_s\in V(R)$ induce $m\geq 1$ edges of weight $1/2$ in $R$, then the product of their edge weights is $\left(\frac 12\right)^m\leq\frac 12$ in $R$ and is 1 in $R_1$.

We have $d_{K_s}(K_r^w)=s!\binom rs\frac 1{r^s}$. Write $w(B_i)=\sum_{v\in B_i}w(v)$ for the total weight of vertices in $B_i$. By \cref{maclaurin}, we have
\begin{align*}
d_{K_s}(R_0)&=s!\sum_{1\leq i_1<\cdots<i_s\leq a}w(B_{i_1})\cdots w(B_{i_s})\leq s!\binom as\frac 1{a^s},
\quad\text{and}
\\d_{K_s}(R_1)&=\sum_{\substack{v_1,\ldots,v_s\in V(R)\\\text{distinct}}}w(v_1)\cdots w(v_s)\leq s!\binom bs\frac 1{b^s}.
\end{align*}
Thus, setting $f(x)=s!\binom xs x^{-s}=(1-\frac 1x)\cdots(1-\frac{s-1}x)$, it suffices to show that $\frac{f(a)+f(b)}2\leq f(\frac{a+b}2)=d_{K_s}(K_r^w)$.

First, observe that
\[
f'(x)=\sum_{i=1}^{s-1} \frac i{x^2} \prod_{\substack{1\leq j<s,\\j\neq i}} \left(1-\frac jx\right)
\]
and that
\begin{align*}
f''(x) &=\sum_{i=1}^{s-1}\Bigg(-\frac{2i}{x^3} \prod_{\substack{1\leq j<s,\\j\neq i}} \left(1-\frac jx\right)
    +\sum_{\substack{1\leq j<s,\\j\neq i}} \frac{ij}{x^4}\cdot \prod_{\substack{1\leq k<s,\\k\neq i,j}} \left(1-\frac kx\right)\Bigg)
\\&=\sum_{i=1}^{s-1}
	\Bigg(-\frac{2i}{x^3} +\sum_{\substack{1\leq j<s,\\j\neq i}} \frac{ij}{x^4}\left(\frac{x}{x-j}\right)\Bigg)
	\Bigg(\prod_{\substack{1\leq j<s,\\j\neq i}}\left(1-\frac jx\right)\Bigg)
\\&=\sum_{i=1}^{s-1}\frac {2i}{x^3}
	\Bigg(-1 +\sum_{\substack{1\leq j<s,\\j\neq i}}\frac{j}{2(x-j)}\Bigg)
	\Bigg(\prod_{\substack{1\leq j<s,\\j\neq i}}\left(1-\frac jx\right)\Bigg).
\end{align*}
It follows that $f''(x)<0$ when $x\geq\binom s2$: for $s=3$, this can be checked manually, and for $s\geq 4$, this follows from the inequality
\[
\sum_{j=1}^{s-1}\frac j{2(x-j)}\leq\sum_{j=1}^{s-1}\frac jx=\binom s2\frac 1x\leq 1.
\]
Hence, if $a\geq\binom s2$, Jensen's inequality implies
\[
d_{K_s}(R)\leq\frac{f(a)+f(b)}{2}\leq f\left(\frac{a+b}{2}\right)=d_{K_s}(K_r^w)
\]
with equality if and only if $a=b=r$, i.e., $R=K_r^w$.
To see that $a\geq\binom s2$, note that \ref{sup5} implies $b\leq(s-1)a$, yielding $sa\geq a+b=t-1\geq s\binom s2$ as desired.

We now prove (2) using (1). Suppose $t>s^2(s-1)/2+s$ is even. Because
$t-1=a+b=\sum_{i=1}^a(|B_i|+1)$
is odd, there exists some $k$ such that $|B_k|+1$ is odd.
Scaling $R-B_k$ up by a factor of $(1-w(B_k))^{-1}$ yields a weighted graph $R_0$ admitting a $(b',a')$-partition, where $a'=a-1$ and $b'=b-|B_k|$.

Set $t'=a'+b'+1=t-|B_k|-1$, which is odd, and set $r=(t'-1)/2$. Let $R'$ be the weighted graph obtained from $R$ by replacing $V(R)-B_k$ with a set $B'$ of $r$ vertices of weight $(1-w(B_k))/r$ each, and setting $w(v',v)=1$ for each $v'\in B'$ and $v\in V(R')\setminus\{v\}$. That is, $R'[B']$ is a copy of $K_{r}^w$ scaled down by a factor of $1-w(B_k)$. 
We note that $R'$ is $\C K_t$-free. Indeed, if $(S_1,S_2)$ is a weighted clique in $R'$, then $S_2$ contains at most one vertex from $B_k$, so $|S_2|\leq r+1$. Hence, $|S_1|+|S_2|\leq |V(R')|+r+1=|B_k|+2r+1=t-1$.

Observe that $|B_k|\leq s-1$ by \ref{sup5}, so $t'>s^2(s-1)/2$. By part (1), we have
\[d_{K_m}(R-B_k)=(1-w(B_k))^md_{K_m}(R_0)\leq(1-w(B_k))^md_{K_m}(K_r^w)=d_{K_m}(R'[B'])\]
for all $m\leq s$. Equality holds if and only if $R_0=K_r^w$, i.e., if and only if $|B_i|=1$ for each $i\neq k$. Hence, if $|B_j|>1$ for some $j\neq k$, we have
\[
d_{K_s}(R)=\sum_{m=0}^s\binom smd_{K_m}(R-B_k)d_{K_{s-m}}R[B_k]
<\sum_{m=0}^s\binom sm d_{K_m}(R'[B'])d_{K_{s-m}}(R'[B_k])
=d_{K_s}(R'),
\]
contradicting the assumption that $d_{K_s}(R)=\pi_s(\C K_t)$. To complete the proof of (2), we recall that $|B_k|$ is even and $|B_k|\leq|B_i|+1$ for any $i\neq k$ by \ref{sup3}. Hence, if $|B_i|=1$ for all $i\neq k$ then $|B_k|=2$.	
\end{proof}

\subsection{The Remaining Positive Results}\label{sec:case 2}

We now prove \cref{conj-wg} in the remaining cases described in \cref{main 2}.
We remark that the cases $t=s+2$ and $s=3$ were proven in \cite{balogh2017on}. However, the proofs of these cases are very straightforward when framed in terms of weighted graphs, so we present them for completeness.

\begin{lemma}
Fix $s\geq 3$ and let $t=s+2$. If $R$ is a $\C K_t$-free weighted graph $R$ with $d_{K_s}(R)=\pi_s(\C K_t)$ of minimum cardinality, it admits an $(s,1)$-partition.
\end{lemma}

\begin{proof}
From \cref{thm:weight:main}\ref{sup2}, $R$ admits a $(b,a)$ partition for parameters $a\geq 1$ and $b\geq s$ satisfying $a+b=t-1=s+1$. It is immediate that $a=1$ and $b=s$.
\end{proof}

In the next two lemmas, we prove \cref{conj-wg} for $t\geq s+3$ and $s=3,4$.

\begin{lemma}
Set $s=3$ and fix $t\geq s+3$. 
Suppose $R$ is a $\C K_t$-free weighted graph with $d_{K_s}(R)=\pi_s(\C K_t)$ of minimum cardinality. Then $R$ admits a $(b,a)$-partition with $b=\max\{s,\lfloor t/2\rfloor\}$.
\end{lemma}

\begin{proof}
By \cref{thm:weight:main}, $R$ admits a $(b,a)$-partition  $B_1\cup\cdots\cup B_a$ such that $a+b=t-1$. Moreover, combining \ref{sup5} with the inequality $t\geq 6$, we conclude $a\geq 2$ and $|B_i|\leq 2$ for each $i$.

To show that $b=\lfloor t/2\rfloor=\max\{s,\lfloor t/2\rfloor\}$ when $t\geq 6$, it suffices to show that all but at most one of the parts $B_i$ have cardinality 1. Equivalently, we must show that $R$ does not contain disjoint edges $uv$ and $xy$ with $w(u,v)=w(x,y)=1/2$.

Suppose the contrary. We may assume without loss of generality that $d_{K_3}(R,\{x,y\})\geq d_{K_3}(R,\{u,v\})$. Let $R'$ be the graph obtained from $R$ by setting $w(u,v)=0$ and $w(x,y)=1$. We observe that $d_{K_3}(R',\{x,y\})=2d_{K_3}(R,\{x,y\})$ and $d_{K_3}(R',\{u,v\})=0$, so
\begin{align*}
d_{K_3}(R')-d_{K_3}(R)
&=d_{K_3}(R',\{x,y\})+d_{K_3}(R',\{u,v\})
	-d_{K_3}(R,\{x,y\})-d_{K_3}(R,\{u,v\})
\\&=d_{K_3}(R,\{x,y\})-d_{K_3}(R,\{u,v\})\geq 0.
\end{align*}
Moreover, $R'$ is $\C K_t$-free, as the clique numbers of $R'_{>\frac 12}$ and $R'_{>0}$ satisfy
\[
\omega(R'_{>\frac 12})+\omega(R'_{>0})\leq\left(\omega(R_{>\frac 12})+1\right)+\left(|V(R)|-1\right)=(a+1)+(b-1)<t.
\]
It follows that $R'$ is also an extremal $\C K_t$-free weighted graph of minimal cardinality. However, this contradicts \cref{thm:weight:main}\ref{sup1}, because $R'$ contains an edge of weight 0. We conclude that $R$ cannot contain disjoint edges $uv$ and $xy$ of weight $1/2$. This implies that all but at most one of the parts $B_i$ have cardinality 1, which is equivalent to showing that $b=\lfloor t/2\rfloor$.
\end{proof}

\begin{lemma}
Set $s=4$ and fix $t\geq s+3$.
Suppose $R$ is a $\C K_t$-free weighted graph with $d_{K_s}(R)=\pi_s(\C K_t)$ of minimum cardinality. Then $R$ admits a $(b,a)$-partition with $b=\max\{s,\lfloor t/2\rfloor\}$.
\end{lemma}

\begin{proof}
By \cref{thm:weight:main}, $R$ admits a $(b,a)$-partition $B_1\cup\cdots\cup B_a$ such that $a+b=t-1$. Combining \ref{sup5} with the inequality $t\geq 7$, we conclude that $a\geq 2$ and that each part has cardinality at most 3. If $t=7$, then we must have $a=2$ and $b=4$, because \ref{sup2} implies $b\geq 4$. Henceforth, we assume $t\geq 8$ and show that $b=\lfloor t/2\rfloor$. Equivalently, we must show that at most one of the parts $B_1,\ldots,B_a$ has cardinality greater than 1.

Order the parts such that
$3\geq |B_1|\geq|B_2|\geq\cdots\geq|B_a|\geq|B_1|-1$; the last inequality is a consequence of \ref{sup3}. Suppose for the sake of contradiction that $|B_2|\geq 2$. We split our proof into three cases, based on $|B_1|$ and $|B_2|$. In each case, we derive a contradiction by constructing a $\C K_t$-free weighted graph with larger $K_s$-density than $R$; these constructions are given in Figure \ref{fig:s=4}.

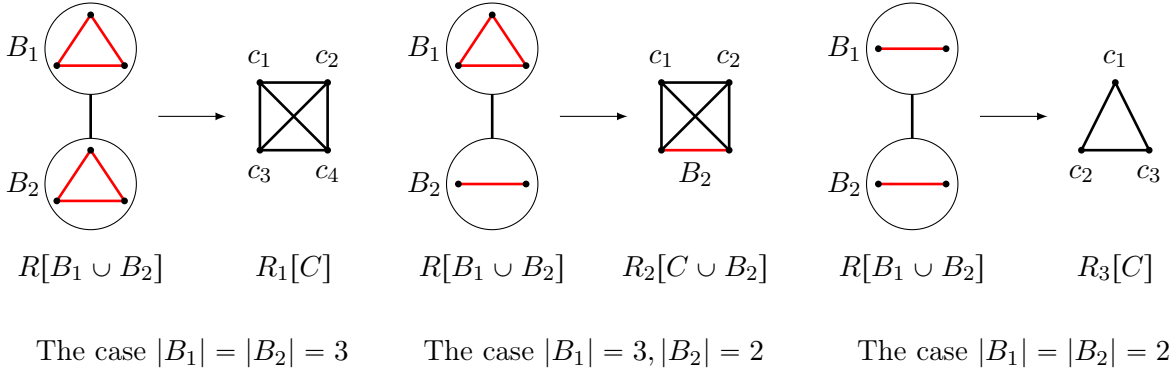
\begin{figure}[h]
\centering    
	\begin{tikzpicture}[scale=0.45]
	\draw[red] [line width=1pt] (3,20.5)-- (5,20.5);
	\draw[red] [line width=1pt] (3,20.5)-- (4,22);
	\draw[red] [line width=1pt] (4,22)-- (5,20.5);
	\fill (3,20.5) circle (0.1cm);
	\fill (5,20.5) circle (0.1cm);
	\fill (4,22) circle (0.1cm);
	\draw (4,21) circle (1.35cm);
	
	\draw[red] [line width=1pt] (3,16.5)-- (5,16.5);
	\draw[red] [line width=1pt] (3,16.5)-- (4,18);
	\draw[red] [line width=1pt] (4,18)-- (5,16.5);
	\fill (3,16.5) circle (0.1cm);
	\fill (5,16.5) circle (0.1cm);
	\fill (4,18) circle (0.1cm);
	\draw (4,17) circle (1.35cm);
	\draw[black] [line width=1pt] (4,19.65)-- (4,18.35);
	\draw (2,21) node[anchor=center] {$B_1$};
	\draw (2,17) node[anchor=center] {$B_2$};
	\draw (4,14.5) node[anchor=center] {$R[B_1\cup B_2]$};
	\draw[-latex] (6,19) -- (8,19); 
	
	\draw[black] [line width=1pt] (9,18)-- (9,20);
	\draw[black] [line width=1pt] (9,18)-- (11,18);
	\draw[black] [line width=1pt] (9,18)-- (11,20);
	\draw[black] [line width=1pt] (9,20)-- (11,18);
	\draw[black] [line width=1pt] (9,20)-- (11,20);
	\draw[black] [line width=1pt] (11,18)-- (11,20);
	\fill (9,18) circle (0.1cm);
	\fill (9,20) circle (0.1cm);
	\fill (11,18) circle (0.1cm);
	\fill (11,20) circle (0.1cm);
	\draw (9,17.3) node[anchor=center] {$c_3$};
	\draw (9,20.7) node[anchor=center] {$c_1$};
	\draw (11,17.3) node[anchor=center] {$c_4$};
	\draw (11,20.7) node[anchor=center] {$c_2$};
	\draw (10,14.5) node[anchor=center] {$R_1[C]$};
	\draw (7,12) node[anchor=center] {The case $|B_1|=|B_2|=3$};
\end{tikzpicture}
\hspace{0.7em}
\begin{tikzpicture}[scale=0.45]
	\draw[red] [line width=1pt] (3,20.5)-- (5,20.5);
	\draw[red] [line width=1pt] (3,20.5)-- (4,22);
	\draw[red] [line width=1pt] (4,22)-- (5,20.5);
	\fill (3,20.5) circle (0.1cm);
	\fill (5,20.5) circle (0.1cm);
	\fill (4,22) circle (0.1cm);
	\draw (4,21) circle (1.35cm);
	\draw[red] [line width=1pt] (3,17)-- (5,17);
	\fill (3,17) circle (0.1cm);
	\fill (5,17) circle (0.1cm);
	\draw (4,17) circle (1.35cm);
	\draw[black] [line width=1pt] (4,19.65)-- (4,18.35);
	\draw (2,21) node[anchor=center] {$B_1$};
	\draw (2,17) node[anchor=center] {$B_2$};
	\draw (4,14.5) node[anchor=center] {$R[B_1\cup B_2]$};
	\draw[-latex] (6,19) -- (8,19); 
	\draw[black] [line width=1pt] (9,18)-- (9,20);
	\draw[red] [line width=1pt] (9,18)-- (11,18);
	\draw[black] [line width=1pt] (9,18)-- (11,20);
	\draw[black] [line width=1pt] (9,20)-- (11,18);
	\draw[black] [line width=1pt] (9,20)-- (11,20);
	\draw[black] [line width=1pt] (11,18)-- (11,20);
	\fill (9,18) circle (0.1cm);
	\fill (9,20) circle (0.1cm);
	\fill (11,18) circle (0.1cm);
	\fill (11,20) circle (0.1cm);
	\draw (10,17.3) node[anchor=center] {$B_2$};
	\draw (9,20.7) node[anchor=center] {$c_1$};
	\draw (11,20.7) node[anchor=center] {$c_2$};
	\draw (10,14.5) node[anchor=center] {$R_2[C\cup B_2]$};
	\draw (7,12) node[anchor=center] {The case $|B_1|=3,|B_2|=2$};
\end{tikzpicture}
\hspace{0.7em}
\begin{tikzpicture}[scale=0.45]
	\draw[red] [line width=1pt] (3,21)-- (5,21);
	\fill (3,21) circle (0.1cm);
	\fill (5,21) circle (0.1cm);
	\draw (4,21) circle (1.35cm);
	\draw[red] [line width=1pt] (3,17)-- (5,17);
	\fill (3,17) circle (0.1cm);
	\fill (5,17) circle (0.1cm);
	\draw (4,17) circle (1.35cm);
	\draw[black] [line width=1pt] (4,19.65)-- (4,18.35);
	\draw (2,21) node[anchor=center] {$B_1$};
	\draw (2,17) node[anchor=center] {$B_2$};
	\draw (4,14.5) node[anchor=center] {$R[B_1\cup B_2]$};
	\draw[-latex] (6,19) -- (8,19); 
	\draw[black] [line width=1pt] (9,18)-- (11,18);
	\draw[black] [line width=1pt] (9,18)-- (10,20);
	\draw[black] [line width=1pt] (10,20)-- (11,18);
	\fill (9,18) circle (0.1cm);
	\fill (10,20) circle (0.1cm);
	\fill (11,18) circle (0.1cm);
	\draw (10,20.7) node[anchor=center] {$c_1$};
	\draw (9,17.3) node[anchor=center] {$c_2$};
	\draw (11,17.3) node[anchor=center] {$c_3$};
	\draw (10,14.5) node[anchor=center] {$R_3[C]$};
	\draw (7,12) node[anchor=center] {The case $|B_1|=|B_2|=2$};
\end{tikzpicture}
\caption{The optimizations in the proof of $s=4$. Red edges have weight $1/2$ and black edges have weight 1.}
\label{fig:s=4}
\end{figure}

Case 1: $|B_1|=|B_2|=3$. In this case, the six vertices in $B_1$ and $B_2$ have the same weight $p$.
Let $R_1$ be the weighted graph obtained from $R$ by replacing $B_1\cup B_2$ with a set $C=\{c_1,c_2,c_3,c_4\}$ of four vertices with weight $3p/2$ each such that $w(c_i,c_j)=w(c_i,v)=1$ for any $i,j\in[4]$ and $v\in B_3\cup\cdots\cup B_a$. We note that  $(R_1)_{>\frac 12}$ and $(R_1)_{>0}$ have clique numbers satisfying
\[
\omega((R_1)_{>\frac 12})+\omega((R_1)_{>0})\leq\left(\omega(R_{>\frac 12})+2\right)+\left(|V(R)|-2\right)=(a+2)+(b-2)<t,
\]
so $R_1$ is $\C K_t$-free. Additionally, one computes that
\begin{align*}
d_{K_0}(R[B_1\cup B_2])&=d_{K_0}(R_1[C])=1,
\qquad\quad d_{K_1}(R[B_1\cup B_2])=d_{K_1}(R_1[C])=6p,
\\d_{K_2}(R[B_1\cup B_2])&=p^2\left(18+12\cdot\frac 12\right)
	<\left(\frac{3p}2\right)^2\cdot 12=d_{K_2}(R_1[C]),
\\d_{K_3}(R[B_1\cup B_2])
&=p^3\left(18\cdot \frac 12+2\cdot\left(\frac 12\right)^3\right)\cdot 3!
	<\left(\frac{3p}2\right)^3\cdot 24=d_{K_3}(R_1[C]),
\\d_{K_4}(R[B_1\cup B_2])&=p^4\left(9\cdot\left(\frac 12\right)^2+6\cdot\left(\frac 12\right)^3\right)\cdot 4!
	<\left(\frac{3p}2\right)^4\cdot 24=d_{K_4}(R_1[C]).
\end{align*}
We conclude that
\[
d_{K_4}(R_1)-d_{K_4}(R)=\sum_{m=0}^4\binom 4m\left(d_{K_m}(R_1[C])-d_{K_m}(R[B_1\cup B_2])\right)d_{K_{4-m}}(R-(B_1\cup B_2))>0,
\]
contradicting the extremality of $R$.

Case 2: $|B_1|=3$ and $|B_2|=2$. Let $p$ and $q$ be the weights of vertices in $B_1$ and $B_2$ respectively; by \ref{sup4}, we have $p\leq q$. Let $R_2$ be the weighted graph obtained from $R$ by replacing $B_1$ with a set $C=\{c_1,c_2\}$ of two vertices with weight $3p/2$ each such that $w(c_1,c_2)=w(c_i,v)=1$ for any $i\in[2]$ and $v\in B_2\cup\cdots\cup B_a$. We note that  $(R_2)_{>\frac 12}$ and $(R_2)_{>0}$ have clique numbers satisfying
\[
\omega((R_2)_{>\frac 12})+\omega((R_2)_{>0})\leq\left(\omega(R_{>\frac 12})+1\right)+\left(|V(R)|-1\right)=(a+1)+(b-1)<t,
\]
so $R_2$ is $\C K_t$-free. One computes that
\begin{align*}
d_{K_0}(R[B_1])&=d_{K_0}(R_2[C])=d_{K_0}(R[B_2])=1,
\\d_{K_1}(R[B_1])&=d_{K_1}(R_2[C])=3p,
&d_{K_1}(R[B_2])&=2q,
\\d_{K_2}(R[B_1])&=p^2\cdot 6\cdot\frac 12
	<\left(\frac {3p}2\right)^2\cdot 2
	=d_{K_2}(R_2[C]),
&d_{K_2}(R[B_2])&=q^2\cdot 2\cdot\frac 12,
\\d_{K_3}(R[B_1])&=p^3\cdot 6\cdot\left(\frac 12\right)^3.
\end{align*}
We now claim that
\[
d_{K_m}(R_2[C\cup B_2])-d_{K_m}(R[B_1\cup B_2])=\sum_{r=0}^m\binom mr(d_{K_r}(R_2[C])-d_{K_r}(R[B_1]))d_{K_{m-r}}(R[B_2])
\]
is positive for $2\leq m\leq 4$. This is immediate for $m=2$. For $m=3,4$, only the $r=3$ term is negative, and one may check that the sum of the $r=2$ and $r=3$ terms is positive via the relations
\[
d_{K_3}(R[B_1])-d_{K_3}(R_2[C])=\frac p2\left(d_{K_2}(R_2[C])-d_{K_2}(R[B_1])\right)
\]
and $d_{K_0}(R[B_2])\leq\frac 2p d_{K_1}(R[B_2])\leq\left(\frac2p\right)^2d_{K_2}(R[B_2])$.
Thus, $d_{K_m}(R_2[C\cup B_2])>d_{K_m}(R[B_1\cup B_2])$ for $2\leq m \leq 4$. It follows that
\[
d_{K_4}(R_2)-d_{K_4}(R)=\sum_{m=0}^4\binom 4m\left(d_{K_m}(R_2[C\cup B_2])-d_{K_m}(R[B_1\cup B_2])\right)d_{K_{4-m}}(R-(B_1\cup B_2))>0,
\]
contradicting the extremality of $R$.

Case 3: $|B_1|=|B_2|=2$. In this case, the four vertices in $B_1$ and $B_2$ have the same weight $p$. Moreover, we have $a\geq 3$ because $t\geq 8$. Let $R_3$ be the weighted graph obtained from $R$ by replacing $B_1\cup B_2$ with a set $C=\{c_1,c_2,c_3\}$ of three vertices with weight $4p/3$ each such that $w(c_i,c_j)=w(c_i,v)=1$ for any $i,j\in[3]$ and $v\in B_3\cup\cdots\cup B_a$. We note that  $(R_3)_{>\frac 12}$ and $(R_3)_{>0}$ have clique numbers satisfying
\[
\omega((R_3)_{>\frac 12})+\omega((R_3)_{>0})\leq\left(\omega(R_{>\frac 12})+1\right)+\left(|V(R)|-1\right)=(a+1)+(b-1)<t,
\]
so $R_3$ is $\C K_t$-free. One computes that
\begin{align*}
d_{K_0}(R[B_1\cup B_2])&=d_{K_0}(R_3[C])=1,
\qquad\qquad d_{K_1}(R[B_1\cup B_2])=d_{K_1}(R_3[C])=4p,
\\d_{K_2}(R[B_1\cup B_2])&=p^2\left(8+4\cdot\frac 12\right)
	<\left(\frac{4p}3\right)^2\cdot 6=d_{K_2}(R_3[C]),
\\d_{K_3}(R[B_1\cup B_2])&=p^3\cdot 24\cdot\frac 12
	<\left(\frac{4p}3\right)^3\cdot 6=d_{K_3}(R_3[C]),
\\d_{K_4}(R[B_1\cup B_2])&=p^4\cdot 24\cdot\left(\frac 12\right)^2.
\end{align*}
We now claim that
\[
d_{K_m}(R_3[C\cup B_3])-d_{K_m}(R[B_1\cup B_2\cup B_3])=\sum_{r=0}^m\binom mr(d_{K_r}(R_3[C])-d_{K_r}(R[B_1\cup B_2]))d_{K_{m-r}}(R[B_3])
\]
is positive for $2\leq m\leq 4$. This is immediate for $m=2,3$. For $m=4$, only the $r=4$ term is negative. By~\ref{sup4}, we have $d_{K_1}(R[B_3])=\sum_{v\in B_3}w(v)\geq p$, so
\begin{align*}
\binom 43(d_{K_3}(R_3[C])-d_{K_3}(R[B_1\cup B_2]))d_{K_1}(R[B_3])
&\geq 4\cdot \left(\frac{10}{27p}d_{K_4}(R[B_1\cup B_2])\right)\cdot p
\\&>\binom 44 d_{K_4}(R[B_1\cup B_2])d_{K_0}(R[B_3]),
\end{align*}
and thus $d_{K_m}(R_3[C\cup B_3])>d_{K_m}(R[B_1\cup B_2\cup B_3])$ for $2\leq m \leq 4$. Therefore, setting $B=B_1\cup B_2\cup B_3$, we have
\[
d_{K_4}(R_3)-d_{K_4}(R)=\sum_{m=0}^4\binom 4m\left(d_{K_m}(R_3[C\cup B_3])-d_{K_m}(R[B])\right)d_{K_{4-m}}(R-B)>0,
\]
contradicting the extremality of $R$.
\end{proof}

\subsection{Counterexamples to Conjecture~\ref{conj}}\label{Counter}\label{sec:case 3}
We conclude this section by presenting some counterexamples to \cref{conj} when $t$ is slightly larger than $2s$. We begin with two counterexamples in the $s=5$ case, then use the same ideas to derive a family of counterexamples for all sufficiently large $s$ and any $t$ with $2s\leq t\leq 2.08s$.

We first observe that Conjecture \ref{conj} is not true for $s=5$ and $t\in \{10,11\}$. 

\begin{figure}[h]
    \centering
\begin{tikzpicture}[scale=0.7]
\draw[red] [line width=1.5pt] (4.5,11)-- (4,10);
\draw [line width=1pt] (4.5,11)-- (6,10);
\draw [line width=1pt] (4.5,11)-- (4.5,9);
\draw [line width=1pt] (4.5,11)-- (5.5,11);
\draw [line width=1pt] (4.5,11)-- (5.5,9);
\draw[red] [line width=1.5pt] (5.5,11)-- (6,10);
\draw [line width=1pt] (5.5,11)-- (5.5,9);
\draw [line width=1pt] (5.5,11)-- (4.5,9);
\draw [line width=1pt] (5.5,11)-- (4,10);
\draw [line width=1pt] (4,10)-- (6,10);
\draw [line width=1pt] (4,10)-- (4.5,9);
\draw [line width=1pt] (4,10)-- (5.5,9);
\draw [line width=1pt] (6,10)-- (4.5,9);
\draw [line width=1pt] (6,10)-- (5.5,9);
\draw[red] [line width=1.5pt] (4.5,9)-- (5.5,9);
\draw [fill=white, line width=1pt] (4,10) circle (0.3cm);
\draw [fill=white, line width=1pt] (6,10) circle (0.3cm);
\draw [fill=white, line width=1pt] (4.5,9) circle (0.3cm);
\draw [fill=white, line width=1pt] (5.5,9) circle (0.3cm);
\draw [fill=white, line width=1pt] (4.5,11) circle (0.3cm);
\draw [fill=white, line width=1pt] (5.5,11) circle (0.3cm);
\draw (7.1,10) node[anchor=center] {$>$};
\draw [line width=1pt] (9,11)-- (8.2,10);
\draw [line width=1pt] (9,11)-- (9.8,10);
\draw [line width=1pt] (9,11)-- (8.5,9);
\draw [line width=1pt] (9,11)-- (9.5,9);
\draw [line width=1pt] (8.2,10)-- (9.8,10);
\draw [line width=1pt] (8.2,10)-- (8.5,9);
\draw [line width=1pt] (8.2,10)-- (9.5,9);
\draw [line width=1pt] (9.8,10)-- (8.5,9);
\draw [line width=1pt] (9.8,10)-- (9.5,9);
\draw[red] [line width=1.5pt] (8.5,9)-- (9.5,9);
\draw [fill=white, line width=1pt] (8.2,10) circle (0.3cm);
\draw [fill=white, line width=1pt] (9.8,10) circle (0.3cm);
\draw [fill=white, line width=1pt] (8.5,9) circle (0.3cm);
\draw [fill=white, line width=1pt] (9.5,9) circle (0.3cm);
\draw [fill=white, line width=1pt] (9,11) circle (0.3cm);
\draw (7.1,7.5) node[anchor=center] {The case $s=5,t=10$};
\end{tikzpicture}
\hspace{3em}
\begin{tikzpicture}[scale=0.7]
\draw[red] [line width=1.5pt] (4.5,11)-- (4,10);
\draw [line width=1pt] (4.5,11)-- (6,10);
\draw [line width=1pt] (4.5,11)-- (4.5,9);
\draw [line width=1pt] (4.5,11)-- (5.5,11);
\draw [line width=1pt] (4.5,11)-- (5.5,9);
\draw[red] [line width=1.5pt] (5.5,11)-- (6,10);
\draw [line width=1pt] (5.5,11)-- (5.5,9);
\draw [line width=1pt] (5.5,11)-- (4.5,9);
\draw [line width=1pt] (5.5,11)-- (4,10);
\draw [line width=1pt] (4,10)-- (6,10);
\draw [line width=1pt] (4,10)-- (4.5,9);
\draw [line width=1pt] (4,10)-- (5.5,9);
\draw [line width=1pt] (6,10)-- (4.5,9);
\draw [line width=1pt] (6,10)-- (5.5,9);
\draw [line width=1pt] (4.5,9)-- (5.5,9);
\draw [fill=white, line width=1pt] (4,10) circle (0.3cm);
\draw [fill=white, line width=1pt] (6,10) circle (0.3cm);
\draw [fill=white, line width=1pt] (4.5,9) circle (0.3cm);
\draw [fill=white, line width=1pt] (5.5,9) circle (0.3cm);
\draw [fill=white, line width=1pt] (4.5,11) circle (0.3cm);
\draw [fill=white, line width=1pt] (5.5,11) circle (0.3cm);
\draw (7.1,10) node[anchor=center] {$>$};
\draw [line width=1pt] (9,11)-- (8.2,10);
\draw [line width=1pt] (9,11)-- (9.8,10);
\draw [line width=1pt] (9,11)-- (8.5,9);
\draw [line width=1pt] (9,11)-- (9.5,9);
\draw [line width=1pt] (8.2,10)-- (9.8,10);
\draw [line width=1pt] (8.2,10)-- (8.5,9);
\draw [line width=1pt] (8.2,10)-- (9.5,9);
\draw [line width=1pt] (9.8,10)-- (8.5,9);
\draw [line width=1pt] (9.8,10)-- (9.5,9);
\draw [line width=1pt] (8.5,9)-- (9.5,9);
\draw [fill=white, line width=1pt] (8.2,10) circle (0.3cm);
\draw [fill=white, line width=1pt] (9.8,10) circle (0.3cm);
\draw [fill=white, line width=1pt] (8.5,9) circle (0.3cm);
\draw [fill=white, line width=1pt] (9.5,9) circle (0.3cm);
\draw [fill=white, line width=1pt] (9,11) circle (0.3cm);
\draw (7.1,7.5) node[anchor=center] {The case $s=5,t=11$};
\end{tikzpicture}
    \caption{Counterexamples to \cref{conj} for $s=5$ and $t\in \{10,11\}$. Red edges have weight $1/2$ and black edges have weight 1.}
    \label{fig:s=5}
\end{figure}
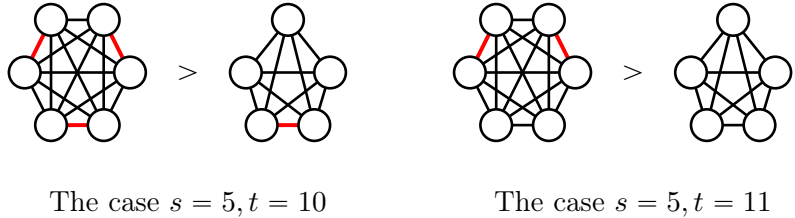

For the case $s=5$ and $t=10$, \cref{conj} hypothesizes that $\pi_s(\C K_t)$ is attained by a weighted graph $R_1$ of order 5 with exactly one edge of weight $1/2$, such that the two vertices incident to the edge of weight $1/2$ have the same weight $p$ and the remaining three vertices have the same weight $q$. Let $R_2$ be a weighted graph of order 6 in which every vertex has weight $1/6$, three disjoint edges have weight $1/2$, and all remaining edges have weight 1. It is straightforward to check that $R_2$ is $\C K_{10}$-free and that
\begin{eqnarray*}
\frac{d_{K_5}(R_2)}{d_{K_5}(R_1)}\ge \frac{\binom 65(\frac{1}{6})^5\frac{1}{4}}{\max\left\{\frac{p^2q^3}{2}: 2p+3q=1\right\}}>1.
\end{eqnarray*}
 Thus Conjecture \ref{conj} is not true in the case $s=5$, $t=10$.
 
 For the case $s=5$ and $t=11$, \cref{conj} hypothesizes that $\pi_s(\C K_t)$ is attained by the complete balanced weighted graph $K_5^w$, which has 5 vertices of weight $1/5$ each and has all edge weights equal to 1. Let $R'$ be a weighted graph of order 6 with two disjoint edges of weight $1/2$ and all other edge weights equal to 1, such that the four vertices incident to the weight-$1/2$ edges have weight $p$ and the remaining two vertices have weight $q$. It is straightforward to check that $R'$ is $\C K_{11}$-free. Moreover, if the parameters $p,q$ are optimized, we have
 \begin{eqnarray*}
\frac{d_{K_s}(R')}{d_{K_s}(K_5^w)}=\frac{\max\left\{\frac{4p^3q^2}{2}+\frac{2p^4q}{4}: 4p+2q=1\right\}}{(\frac{1}{5})^5}>1.
\end{eqnarray*}
Notably, the inequality holds with $p=0.16$ and $q=0.18$. Thus \cref{conj} is also false in the case $s=5$, $t=11$.

We now show that a similar construction works if $s$ is sufficiently large.

\begin{lemma}
\cref{conj} is false for all sufficiently large $s$ and any $t$ satisfying $2s\leq t\leq 2.08s$.
\end{lemma}

\begin{proof}
First suppose $t$ is odd, i.e., $t=2r+1$ for some integer $r\geq s$. \cref{conj} hypothesizes that $\pi_s(\C K_t)$ is attained by the complete balanced weighted graph $K_r^w$, which has $r$ vertices of weight $1/r$ each and has all edge weights equal to 1. Let $R'$ be a weighted graph on $(r+1)$ vertices with two disjoint edges of weight $1/2$ and all other edge weights equal to 1, such that the four vertices incident to weight-$1/2$ edges have weight $3/4r$ and the remaining $r-3$ vertices have weight $1/r$. It is straightforward to check that $R'$ is $\C K_t$-free. Moreover, because $s\leq r\leq 1.04s$, we have
\begin{eqnarray*}
\frac{d_{K_s}(R')}{d_{K_s}(K_r^w)}\ge\frac{\binom{r+1}{s}\cdot (\frac{1}{r})^{s-4}(\frac{3}{4r})^4 \frac{1}{4}}{\binom{r}{s}(\frac{1}{r})^s}
= \frac{r+1}{r+1-s}\cdot \frac{3^4}{4^5}
\geq\frac{r+1}{0.04r+1}\cdot 0.079
>1
\end{eqnarray*}
if $s\leq r$ is sufficiently large.

Next, suppose $t$ is even, i.e., $t=2r$ for some integer $r\geq s$. \cref{conj} hypothesizes that $\pi_s(\C K_t)$ is attained by a weighted graph $R_1$ on $r$ vertices with exactly one edge of weight $1/2$. By \cref{maclaurin}, we have that 
\[
d_{K_s}(R_1)\leq \sum_{\substack{v_1,\ldots,v_s\in V(R)\\\text{distinct}}}w(v_1)\cdots w(v_s)
\leq s!\binom rs\left(\frac 1r\right)^s=d_{K_s}(K_r^w).
\]
Let $R_2$ be a weighted graph on $(r+1)$ vertices with three disjoint edges of weight $1/2$ and all other edge weights equal to 1, such that the six vertices incident to weight-1/2 edges have weight $5/6r$ and the remaining $r-5$ vertices have weight $1/r$. It is straightforward to check that $R_2$ is $\C K_t$-free. Moreover, because $s\leq r\leq 1.04s$, we have
\begin{eqnarray*}
\frac{d_{K_s}(R_2)}{d_{K_s}(R_1)}\ge \frac{d_{K_s}(R_2)}{d_{K_s}(K_r^w)} \ge \frac{\binom{r+1}{s}\cdot (\frac{1}{r})^{s-6}(\frac{5}{6r})^6 \frac{1}{8}}{\binom{r}{s}(\frac{1}{r})^s}
= \frac{r+1}{r+1-s}\cdot \frac{5^6}{6^6}\cdot \frac{1}{8}
\geq \frac{r+1}{0.04r+1}\cdot 0.041>1
\end{eqnarray*}
if $s\leq r$ is sufficiently large.
\end{proof}

\section{Concluding Remarks}\label{sec:concluding remarks}
In this paper, we combinatorially resolve the generalized Ramsey--Tur\'an problem for cliques, reducing its determination to a bounded optimization problem about finding the optimal $(b,a)$-partition, which remains an intriguing problem.

\begin{problem}
    Given integers $t-2\geq s\ge 3$, which $(b,a)$-partition with $a+b=t-1$ achieves the Ramsey--Tur\'an density $\varrho_s(K_t)$?
\end{problem}

An easier, yet still interesting, problem is the following. By~\cref{main 2}, the threhold value of $t$ for the extremal periodic behavior lies somewhere between $2.08s$ and $s^3$. Which bound is closer to the truth?

For general graphs, an Erd\H{o}s--Stone--Simonovits type result is still out of reach. For example, we do not know whether $\varrho_2(K_{2,2,2})>0$~\cite{erdo1983more, simonovits2001ramsey}. In light of this, we wonder the following.

\begin{problem}
    Decide if $\varrho_3(K_{2,2,2,2})>0$ or not.
\end{problem}

Another natural future direction is to study $\mathrm{RT}(n,K_s,K_t,f(n))$ for smaller independence numbers, e.g.\ when $f(n)=n^{1-\varepsilon}$ or when it is the inverse function of the Ramsey function, say $f(n)=\sqrt{n\log n}$.
\\

\noindent\textbf{Note.} After this paper was written, we learned that Balogh, Magnan and Palmer \cite{Balogh2024Generalized} independently proved some related results.

\bibliography{reference.bib}

\end{document}